\documentclass[a4paper,11pt,twoside]{article}
\usepackage[utf8]{inputenc}
\usepackage[T1]{fontenc}
\usepackage[english]{babel}
\usepackage{amssymb}
\usepackage[dvipsnames]{xcolor}
\usepackage{amsmath}
\usepackage{amsfonts}
\usepackage{amsthm}
\usepackage{fancyhdr}
\usepackage{esint}
\usepackage{authblk}
\usepackage{multicol}
\usepackage[cm]{fullpage}

\usepackage{hyperref}
\usepackage{enumerate}

\newcommand{\R}{\mathbb{R}}

\newtheorem{thm}{Theorem}[section]

\newtheorem{ppt}{Proposition}[section]
\newtheorem{defi}{Definition}[section]
\newtheorem{lem}{Lemma}[section]
\newtheorem{remark}{Remark}
\numberwithin{equation}{section}

\author[1]{Laurent CHUPIN, Nicolae CÎNDEA and Geoffrey LACOUR \thanks{Corresponding author: Geoffrey Lacour - \texttt{geoffrey.lacour@uca.fr}}}
\affil[1]{Université Clermont Auvergne, CNRS, LMBP, F-63000 Clermont-Ferrand, France}

\title{Variational inequality solutions and finite stopping time for a class of shear-thinning flows}
\date{}

\begin{document}

\maketitle

 
 
\begin{abstract} 
 
The aim of this paper is to study the existence of a finite stopping time for solutions in the form of variational inequality to fluid flows following a power law (or Ostwald-DeWaele law) in dimension $N \in \{2,3\}$. We first establish the existence of solutions for generalized Newtonian flows, valid for viscous stress tensors associated with the usual laws such as Ostwald-DeWaele, Carreau-Yasuda, Herschel-Bulkley and Bingham, but also for cases where the viscosity coefficient satisfies a more atypical (logarithmic) form. To demonstrate the existence of such solutions, we proceed by applying a nonlinear Galerkin method with a double regularization on the viscosity coefficient. We then establish the existence of a finite stopping time for threshold fluids or shear-thinning power-law fluids, i.e. formally such that the viscous stress tensor is represented by a $p$-Laplacian for the symmetrized gradient for $p \in [1,2)$.
\end{abstract}

%%%%%%%%%%%%%%%%%%%%%%%%%%%%%%%%%%%%%%%%%%%%%%%%%%%%%%%%%%%%%%%%%%%%%%%
%%%%%%%%%%%%%%%%%%%%%%%%%%%%%%%%%%%%%%%%%%%%%%%%%%%%%%%%%%%%%%%%%%%%%%%
%%%%%%%%%%%%%%%%%%%%%%%%%%%%%%%%%%%%%%%%%%%%%%%%%%%%%%%%%%%%%%%%%%%%%%%
\section{Introduction}\label{sec:introduction}
%%%%%%%%%%%%%%%%%%%%%%%%%%%%%%%%%%%%%%%%%%%%%%%%%%%%%%%%%%%%%%%%%%%%%%%
%%%%%%%%%%%%%%%%%%%%%%%%%%%%%%%%%%%%%%%%%%%%%%%%%%%%%%%%%%%%%%%%%%%%%%%
%%%%%%%%%%%%%%%%%%%%%%%%%%%%%%%%%%%%%%%%%%%%%%%%%%%%%%%%%%%%%%%%%%%%%%%

The aim of this paper is to establish the existence of a finite stopping time for variational inequality solutions of a flow following an Ostwald-DeWaele, Bingham, or Herschel-Bulkley law in a diffusive setting. Such flows can be formally represented by the following system:

\begin{equation}\label{eq:0}
\begin{cases}
\partial_tu + (u \cdot \nabla) u + \nabla \pi - \Delta u - {\rm{div}}\left(F\left(\lvert D(u) \rvert\right) D(u)\right) = f& \text{in} \; (0,+\infty)\times \Omega\\
{\rm{div}}(u) = 0& \text{in}\; (0,+\infty) \times \Omega\\
u = 0& \text{on}\; [0,+\infty) \times \partial\Omega\\
u = u_0& \text{on}\; \{0\} \times \Omega,
\end{cases}
\end{equation}

where $\Omega$ is an open bounded subset of $\mathbb{R}^N$, for $N \in
\{2,3\}$ with a regular enough boundary $\partial \Omega$. Such
nonlinear systems describe the flow of incompressible generalized Newtonian
fluids and give rise to several relevant models. Several types of fluids are described by~\eqref{eq:0}. Firstly, if $F(t) = C$, the system~\eqref{eq:0} is the Navier-Stokes equations for a viscous incompressible fluid. Also, by choosing $F(t) =(1 + t^2)^{\frac{p-2}{2}}$, system~\eqref{eq:0} describes a Carreau flow. Another relevant example is obtained by choosing, for $p \in (1,2)$ by $F(t) = t^{p-2}$ for $t > 0$ which leads to an Ostwald-DeWaele (power-law) flow. In the particular case of Bingham arising for $p = 1$ which describes a plastic behavior, we get $F(0) \in [0,1]$%, where $\tau^* = 1 > 0$ is the plasticity threshold, which is scaled here
. In the latter case, the function is multivalued at the origin (note that a physical consequence of this phenomenon is the nonexistence of a reference viscosity for threshold fluids, see for example \cite{becker-80}). It is now established that this problem can be circumvented by considering the function outside the origin by a regularization process and by giving a meaning to its limit, in sense of sub-differential. This approach has been successfully carried out in the case of a two-dimensional Bingham flow (see, for instance, \cite{duvaut-lions}). 

In the present paper, we focus on the mathematical analysis of shear-thinning flows: a flow is said to be shear-thinning when its viscosity decreases as a function of the stresses applied to it, namely, in the flows we consider, that the function decreases as the shear rate increases. We mainly refer to \cite{coussot,bird,galdi} for the physical motivations of such models. Throughout the present article, we will consider simple fluid flows, that is we will make the assumption that the shear rate is the second invariant of the strain-rate tensor, and moreover it is a scalar quantity given by $\lvert D(u) \rvert$.

In the non-diffusive case (i.e. without Laplacian) the existence and the regularity of distributional solutions for $p > \frac{2N}{N+2}$, which corresponds to the limiting case of the compact Sobolev embedding $W^{1,p}(\Omega) \hookrightarrow L^2(\Omega)$ is known for the various boundary conditions, both in the stationary and evolutionary cases. We refer for example to \cite{ha,lbldmr,lbmr,fccg,ldmr,ldmrjw,hemr,jfjmms,jfmr,jmjnmr,jmjnkr,jw,jzzt} and the references therein for more details, as well as to the monograph \cite{ChlebickaGwiazdaSwierczewska-GwiazdaWroblewska-Kaminska-2021} for a complete and modern presentation of this type of problem. In this non-diffusive case, it is possible to show (see \cite{BurczakModenaSzekelyhidi-2021}) that the problem can be ill-posed in the sense of distributional solutions in the case $\frac{2N}{N+2} \geq p$. One can avoid such hypotheses on $p \geq 1$ by using dissipative solutions, whose existence has been proved in \cite{abbatiello-feireisl-20} in the three-dimensional setting. For the above reasons, we consider in the present paper variational inequality solutions in a diffusive setting, which we believe particularly interesting in view of numerical simulations perspective (see for example \cite[Chapter 4]{glowinski-lions-tremolieres} or \cite{saramito}) as for controllability (see for example \cite{friedman-86,ito-kunisch-09}).

Secondly, we focus on a remarkable property of shear-thinning power-law type fluids: the existence of a finite stopping time. Such a property has been established, for example, in the case of a two-dimensional Bingham flow in \cite{glowinski}, in the case of some electrorheological fluids in \cite{aafcpm}, as well as for the parabolic $p$-Laplacian operator (see \cite{dibenedetto-degenerate-parabolic}).

Roughly speaking, this property translates into the existence of a time $T_s > 0$ from which the fluid is at a standstill, i.e. such that the velocity field solution to the equation verifies $u(t) = 0$ for almost all $t \geq T_s$. Intuitively, the existence of such a stopping time for the fluid is specific to the shear-thinning character for the Ostwald-DeWaele fluid (as well as for a plastic fluid): the viscosity coefficient given by $F(t) = t^{p-2}$ is decreasing in such a case, which is formally characterized by $1 \leq p < 2$, and amounts to saying that the fluid's viscosity is all the greater the lower the stresses applied to it. It is therefore to be expected that, with no external force adding energy to the system, the time decay of the fluid's energy implies that its viscosity will increase until it stops. Note that in the diffusive case, which we consider here through the system~\eqref{eq:0}, we are able to establish the finite stopping time of the kinetic energy associated with the solution of \eqref{eq:0}, i.e. the stopping of its $L^2$-norm.

Having established the existence of weak solutions in the form of a parabolic variational inequality (see Theorem~\ref{thm:0} and Definition~\ref{defi:0}) for tensors $\tau$ of the form:

\begin{equation*}
\tau(D(u)) = F(| D(u) |)D(u),
\end{equation*}

similar to those established, for example, in \cite{duvaut-lions}, we will establish the existence of such a stopping time for the kinetic energy of solutions via a differential inequality method (see Theorem~\ref{thm:1}).

{\large \textbf{Assumptions over the viscosity coefficient $F$}}

Throughout this article, we will assume that the viscosity coefficient $F$ satisfies the following assumptions:
\begin{enumerate}
\item[(C1)] $F : (0,+\infty) \rightarrow (0,+\infty)$;
\item[(C2)]  $F \in  W^{1,\infty}_{\mathrm{loc}}\left((0,+\infty)\right)$;
\item[(C3)] $t \mapsto tF(t)$ is non-decreasing on $(0,+\infty)$;
\item[(C4)] there exist $p \in \left[1,2\right]$, $t_0 > 0$ and
  $K > 0$ such that for every $t \geq t_0$,  $F(t) \leq Kt^{p-2}$.
\end{enumerate}

 Some examples of functions verifying the above assumptions are given
 in Appendix~\ref{sec:some-exampl-funct}. We emphasize in particular that this takes into account many physical models, such as the Carreau, Bingham, Herschel-Bulkley, Cross, or power law flows. 
 
\begin{remark} Assumption (C3) is equivalent to the fact that for all $\varepsilon \geq 0$, the function $t \mapsto tF\left(\sqrt{\varepsilon + t^2}\right)$ is non-decreasing. Indeed, we can write:
$$\forall t \in (0,+\infty),\; tF\left(\sqrt{\varepsilon + t^2}\right) = \left(\frac{t}{\sqrt{\varepsilon + t^2}}\right)\sqrt{\varepsilon + t^2}F\left(\sqrt{\varepsilon + t^2}\right).$$
Hence, $t \mapsto tF\left(\sqrt{\varepsilon + t^2}\right)$ is the product of two non-negative and non-decreasing functions, so it is a non-decreasing function. The opposite implication being obvious by setting $\varepsilon = 0$. 
\end{remark}

\begin{remark}
Since the main objective of this article is not to study the existence of solutions, we have established Galerkin method by considering the assumption (C2) in order to make use of Picard-Lindelöf theory, but note that it is possible to weaken this hypothesis by making use of Cauchy-Peano or Carathéodory theory. Then, the assumption (C2) can be replace by $F \in C_{\mathrm{loc}}((0,+\infty))$ without changing the proof of Theorem~\ref{thm:0}.
\end{remark}

The existence of a finite stopping time for the kinetic energy associated to variational inequality weak solutions following from Theorem~\ref{thm:0} is proved in Section~\ref{sec:finite-stopping-time}, while considering a viscosity coefficient $F$ verifying (C1)-(C4) and such that it describes a power-type law, namely it satisfies for $1 \leq p < 2$ the additional assumption

  \begin{equation}
    \label{eq:F-power}
    F(t) \geq Ct^{p-2}.
  \end{equation}
  %in~\eqref{eq:0}.%

Let us conclude with the observation that many fluids are described or approximated by such a law, and are used in a wide range of practical applications. Furthermore, many thixotropic flows (such as blood) also fall into this category, depending on the circumstances of the flow studied (see \cite{RobertsonSequeiraOwens-2009}).

\textbf{\underline{Notations:}} Throughout the paper, we denote in a generic way the constants by the letter $C$, and omit their dependence on the parameters in the notations while irrelevant for our study. The functional spaces are defined as follow. We denote by $\mathcal{C}^\infty_{0,\sigma}(\Omega)$ the space of divergence-free functions belonging to the space of smooth and compactly supported functions $\mathcal{C}^\infty_0(\Omega)$, and by $L^2_\sigma(\Omega)$ the closure of $\mathcal{C}^\infty_{0,\sigma}(\Omega)$ in $L^2(\Omega)$. Then, recalling that~$H_0^1(\Omega)$ is the closure of $\mathcal{C}_0^{\infty}(\Omega)$ into~$H^1(\Omega)$ (which is endowed with the norm $u \mapsto \lVert \nabla u \rVert_{L^2}$), we consider the Sobolev space $H_{0,\sigma}^1(\Omega)$ defined as 
\begin{equation*}
    H_{0,\sigma}^1(\Omega) := \left\{ w \in L^2_\sigma(\Omega)/\; w = \nabla v,\; v \in H_0^1(\Omega)\right\},
\end{equation*}
which is composed of functions whose trace and divergence are null. We denote~$H^{-1}_{\sigma}(\Omega)$ its dual and $\langle \cdot, \cdot\rangle$ is the duality product between
$H^{-1}_{\sigma}(\Omega)$ and~$H_{0,\sigma}^1(\Omega)$.  Finally, the space $\mathcal{C}_w(\R_+,L^2_\sigma(\Omega)$ is the functional space whose elements are continuous in the time variable and belonging to $L^2_\sigma(\Omega)$ endowed with its weak topology in the space variable. We should add when necessary the index ``$\mathrm{loc}$'' to underline that we consider local in time solutions.

%%%%%%%%%%%%%%%%%%%%%%%%%%%%%%%%%%%%%%%%%%%%%%%%%%%%%%%%%%%%%%%%%%%%%%%
%%%%%%%%%%%%%%%%%%%%%%%%%%%%%%%%%%%%%%%%%%%%%%%%%%%%%%%%%%%%%%%%%%%%%%%
%%%%%%%%%%%%%%%%%%%%%%%%%%%%%%%%%%%%%%%%%%%%%%%%%%%%%%%%%%%%%%%%%%%%%%%
\section{Weak characterization of solutions by a parabolic variational inequality}\label{sec:weak-char-parab}
%%%%%%%%%%%%%%%%%%%%%%%%%%%%%%%%%%%%%%%%%%%%%%%%%%%%%%%%%%%%%%%%%%%%%%%
%%%%%%%%%%%%%%%%%%%%%%%%%%%%%%%%%%%%%%%%%%%%%%%%%%%%%%%%%%%%%%%%%%%%%%%
%%%%%%%%%%%%%%%%%%%%%%%%%%%%%%%%%%%%%%%%%%%%%%%%%%%%%%%%%%%%%%%%%%%%%%%

In this section we introduce a weak formulation of system~\eqref{eq:0}
using a parabolic variational inequality (see Definition~\ref{defi:0}). First, we point out that in the system~\eqref{eq:0}, we consider a non-slip boundary condition on~$\partial\Omega$.It is thus natural to assume that the initial velocity field~$u_0$ is of
null trace on~$\partial\Omega$, namely~$u_0$ belongs to
$H_{0,\sigma}^1(\Omega)$. Following the ideas employed for showing the existence of solution to
Bingham equations in~\cite{duvaut-lions,stampacchia}, we
define a functional~$j$ making appear the viscous non-linear term in~\eqref{eq:0} in its derivative.

We fix for the moment $0 \le \varepsilon \le \delta $ and we define a function
$G_\varepsilon : (0, +\infty) \to (0,+\infty)$ and a functional
$j_\varepsilon : H_{0, \sigma}^1(\Omega) \to \mathbb{R}$ by
\begin{equation}
\label{eq:G-delta}
G_{\varepsilon}(t) = \int_{0}^{t}s F(\sqrt{\varepsilon
  + s^2})\; ds \qquad \text{for every } t \in (0, +\infty)
\end{equation}
and
\begin{equation}
  \label{eq:J-eps}
  j_\varepsilon (v) = \int_\Omega G_\varepsilon(|D(v)|)\, dx, \quad
  (v \in H_{0, \sigma}^1(\Omega)),
\end{equation}
respectively. We also denote~$j = j_0$ and~$G = G_0$.
One can check that
$G_{\varepsilon}$ is a convex functional for~$\varepsilon$ small enough. Indeed,
\[
  G_{\varepsilon}'(t) = t F(\sqrt{\varepsilon + t^2}), \qquad
  \text{for every } t \in (0, +\infty),
\]
and applying the hypothesis~(C3) the convexity of~$G$ follows
immediately. Moreover, we point out that the functional~$j_\varepsilon$
  defined by~\eqref{eq:J-eps} is convex and verifies
  \begin{equation}
    \label{eq:J-eps-'}
    \langle j_\varepsilon'(v),w \rangle_{-1,1} =
    \int_{\Omega}F\left(\sqrt{\varepsilon + \lvert D(v)
        \rvert^2}\right)\left(D(v):D(w)\right)\; dx \qquad (v, w \in
    H_{0, \sigma}^1(\Omega)).
  \end{equation}

% \begin{lem}\label{lem:0} For every $\varepsilon > 0$ the functional~$j_\varepsilon$
%   defined by~\eqref{eq:J-eps} is convex and verifies
%   \begin{equation}
%     \label{eq:J-eps-'}
%     \langle j_\varepsilon'(v),w \rangle_{-1,1} =
%     \int_{\Omega}F\left(\sqrt{\varepsilon + \lvert D(v)
%         \rvert^2}\right)\left(D(v):D(w)\right)\; dx \qquad (v, w \in
%     H_{0, \sigma}^1(\Omega)).
%   \end{equation}
% \end{lem}

% \begin{proof}
%   The convexity of~$j_\varepsilon$ is immediately obtained from the
%   hypothesis~(C3) and~\eqref{eq:J-eps}.
% For every $t \in \mathbb{R}$ we have
% \begin{align*}
% \frac{d}{dt}\left(G_{\varepsilon}\left(\lvert
%   D(v+tw)\rvert\right)\right) &= G_{\varepsilon}'\left(\lvert D(v+tw)
%                                 \rvert\right) \frac{d}{dt}\left(\lvert D(v+tw) \rvert\right)\\
% &=F\left(\sqrt{\varepsilon + \lvert D(v+tw) \rvert^2} \right)\lvert D(v+tw) \rvert \left(\frac{D(v+tw) : D(w)}{\lvert D(v+tw) \rvert}\right)\\
% &= F\left(\sqrt{\varepsilon + \lvert D(v+tw) \rvert^2} \right) D(v+tw) : D(w).
% \end{align*}

%  Hence
% \begin{align*}
% \langle j_\varepsilon'(v+tw), w\rangle = \frac{d}{dt}j_\varepsilon(v + tw)  &= \int_{\Omega} \frac{d}{dt}\left(G_{\varepsilon}\left(\lvert D(v+tw)\rvert\right)\right)\; dx \\
% & = \int_{\Omega} F\left(\sqrt{\varepsilon + \lvert D(v+tw) \rvert^2} \right) D(v+tw) : D(w)\; dx.
% \end{align*}
% Letting~$t$ going to~$0$ we obtain~\eqref{eq:J-eps-'}.
% \end{proof}

\begin{remark} We point out that~$j'$ is well defined. Firstly, by our assumptions (C2) and (C3), we can deduce that for all $\beta \in \left(0,\frac{1}{2}\right)$, there exists~$\delta_0$ such that:
$$F(t) \leq t^{-(1 + \beta)} \quad \text{for every} \; t \in (0,\delta_0).$$

Indeed, assume that this last inequality does not hold, then for every~$\delta_0 > 0$, there exists~$t_0 \in (0,\delta_0)$ such that:
$$F(t_0) > t_0^{-(1 + \beta)}.$$

We can consider without loss of generality that $\delta_0 < {\rm min}\left(1,F(1)^{-\frac{1}{\beta}}\right)$, which implies, using our assumption~(C3):
 
$$\delta_0^{-\beta} < t_0^{-\beta} < t_0F(t_0) \leq F(1).$$

This contradiction shows the result. We recall Korn's~$L^2$ equality for divergence free vector fields:
$$\int_{\Omega} \lvert D(\varphi) \rvert^2\; dx = \frac{1}{2}\lVert \varphi \rVert^2_{H_0^1}, \quad (\varphi \in H_{0,\sigma}^1(\Omega)).$$
 Using these last results and applying Cauchy Schwarz's and Hölder's inequalities, we get:
\begin{align*}
\lvert \langle j'(u), \varphi \rangle_{-1,1} \rvert &= \left\lvert \int_{\Omega}F(\lvert D(u) \rvert)D(u):D(\varphi)\; dx\right\rvert\\
&\leq \frac{1}{\sqrt{2}}\left(\int_{\Omega} F(\lvert D(u) \rvert)^2\lvert D(u) \rvert^2\; dx\right)^{\frac{1}{2}}\lVert \varphi \rVert_{H_0^1}\\
&=\frac{1}{\sqrt{2}}\left(\int_{\left\{\lvert D(u) \rvert \leq \delta_0\right\}} F(\lvert D(u) \rvert)^2\lvert D(u) \rvert^2\; dx + \int_{\left\{\lvert D(u) \rvert > \delta_0\right\}} F(\lvert D(u) \rvert)^2\lvert D(u) \rvert^2\; dx\right)^{\frac{1}{2}}\lVert \varphi \rVert_{H_0^1}\\
&\leq \frac{1}{\sqrt{2}}\left(\int_{\left\{\lvert D(u) \rvert \leq \delta_0\right\}} \lvert D(u) \rvert^{-2\beta}\; dx + \int_{\left\{\lvert D(u) \rvert > \delta_0\right\}} F(\lvert D(u) \rvert)^2\lvert D(u) \rvert^2\; dx\right)^{\frac{1}{2}}\lVert \varphi \rVert_{H_0^1}\\
& = \frac{1}{\sqrt{2}}\left(\frac{1}{1-2\beta}\int_{\left\{\lvert D(u) \rvert \leq \delta_0\right\}}\int_0^{\lvert D(u) \rvert}s^{1-2\beta}\; ds\; dx + \int_{\left\{\lvert D(u) \rvert > \delta_0\right\}} F(\lvert D(u) \rvert)^2\lvert D(u) \rvert^2\; dx\right)^{\frac{1}{2}}\lVert \varphi \rVert_{H_0^1}.
\end{align*}

This implies that~$j'$ is well-defined.
\end{remark}

We now establish the definition of solutions in the form of variational inequality, which we will consider in the rest of the article.

\begin{defi}[Weak solution of \eqref{eq:0}]\label{defi:0}
  We say that a function
  $u \in L^{2}_{\mathrm{loc}}\left(\R_+,H_{0,\sigma}^1(\Omega)\right) \cap \mathcal{C}_{w,\mathrm{loc}}(\R_+,L^2_{\sigma}(\Omega))$ such that
  $\partial_tu \in L^{\frac{4}{N}}_{\mathrm{loc}}\left(\R_+,H^{-1}_{\sigma}(\Omega)\right)$
  is a weak solution of \eqref{eq:0} if and only if~$u$ verifies $u_{|t=0} = u_0 \in H_{0, \sigma}^1(\Omega)$, and for every fixed~$T > 0$ and all
  $\varphi \in \mathcal{C}^{\infty}((0,T) \times \Omega)$ we have:
  
\begin{align}
  &\int_0^T
    \left\langle \partial_tu(t),\varphi(t)\right
    \rangle\; dt + \frac{1}{2}\left(\lVert u_0 \rVert_{L^2(\Omega)}^2 - \lVert u(T) \rVert_{L^2(\Omega)}^2\right) 
    + \int_0^T\int_{\Omega}
    D(u(t)):D(\varphi(t)- u(t))\; dx \nonumber \\
  &
    -\int_0^T\int_{\Omega}\left(u(t)\cdot \nabla u(t)\right) \cdot
    \varphi(t)\; dx\; dt + \int_0^T\int_{\Omega} G\left(\lvert D(\varphi(t))
    \rvert\right) - G\left(\lvert D(u(t)) \rvert\right)\; dx
     \; dt \nonumber \\
  & \geq \int_0^T\left\langle f(t),\varphi(t) - u(t)\right\rangle\; dt. \label{eq:weak-sol-ineq}
\end{align} 
\end{defi}

 Let us quickly motivate this definition with some formal computations. First, we point out that since~$u$ belongs to $\mathcal{C}_{w,\mathrm{loc}}(\R_+,L^2_{\sigma}(\Omega))$, Definition~\ref{defi:0} makes sense. Then, if we consider for some fixed~$T > 0$ that the
 Lebesgue measure of the set
 \[
   \left\{(t,x) \in (0,T)\times \Omega \ \mid \ \lvert D(u)(t,x) \rvert \leq \delta
   \right\}
 \]
 is equal to zero for a small~$\delta > 0$, we have that:
$$\int_0^T \langle j'(u), \varphi \rangle\, dt = \int_0^T \int_{\Omega} F\left(\lvert D(u) \rvert\right)\left(D(u) : D(\varphi)\right)\; dx\, dt.$$

 Now, if we replace~$\varphi$ by~$u + s\varphi$, with~$s>0$, in the variational inequality \eqref{eq:weak-sol-ineq}, we obtain after dividing by~$s$:

\begin{align*}
&\int_0^T \int_{\Omega} D(u):D(\varphi)\;dx\, dt + \int_0^T\int_{\Omega}\frac{G\left(\lvert D(u + s\varphi) \rvert\right) - G\left(\lvert D(u) \rvert\right)}{s}\;dx\, dt\\
& \geq \int_0^T \int_{\Omega} \langle f- \partial_tu, \varphi \rangle\, dt - \int_0^T\int_{\Omega}\left(u \cdot\nabla u\right) \cdot \varphi\; dx\, dt.
\end{align*}

 Since~$j$ admits a Fréchet-derivative, it also admits a Gâteaux-derivative and both are the same. Hence, taking the limit as~$s \rightarrow 0$:
\begin{align*}
  &\int_0^T\int_{\Omega}D(u):D(\varphi)\;dx\, dt + \int_0^T\int_{\Omega}F\left(\lvert D(u) \rvert\right)\left(D(u):D(\varphi)\right)\;dx\, dt\\
  & \geq \int_0^T\int_{\Omega} \langle f- \partial_tu, \varphi \rangle\, dt
    -
    \int_0^T\int_{\Omega}\left(u \cdot \nabla
    u\right)
    \cdot
    \varphi\;
    dx\, dt.
\end{align*}

Repeating once again the previous reasoning but writing~$u -s\varphi$ instead of~$u+ s\varphi$, we get the following equality:
\begin{align*}
&\int_0^T\int_{\Omega}D(u):D(\varphi)\;dx\, dt + \int_0^T\int_{\Omega}F\left(\lvert D(u) \rvert\right)\left(D(u):D(\varphi)\right)\;dx\, dt\\
& = \int_0^T\int_{\Omega} \langle f- \partial_tu, \varphi \rangle\, dt - \int_0^T\int_{\Omega}\left(u\cdot\nabla u\right)\cdot\varphi\; dx\, dt.
\end{align*}

 Therefore, assuming that~$u$ is regular enough, we obtain
\begin{align*}
&-\frac{1}{2}\int_0^T\int_{\Omega}\Delta u \cdot \varphi\;dx\,dt - \int_0^T\int_{\Omega}\text{div}\left(F\left(\lvert D(u) \rvert\right)D(u)\right)\varphi\;dx\, dt\\
& = \int_0^T\int_{\Omega}\left(f- \partial_tu - u \cdot\nabla u\right)\cdot\varphi\; dx\,dt.
\end{align*}

Furthermore De Rham's theorem for a domain with Lipschitz boundary states that there exists a pressure term~$p$ such that~$f = \nabla p$ into some well chosen space (see \cite[section 2]{ldmrjw} for details). Considering such a function and also the two previous observations, we can write:
$$\int_0^T\int_{\Omega}\left(\partial_tu + u.\nabla u - \frac{1}{2}\Delta u + \nabla p -\text{div}\left(F\left(\lvert D(u) \rvert\right)D(u)\right)
    - f\right)\varphi\; dx\,dt = 0,
    \qquad
    \left(\varphi \in \mathcal{C}^{\infty}((0,T) \times \Omega)\right),$$
    
which is almost everywhere equivalent to the equation \eqref{eq:0} up to the multiplicative dynamic viscosity constant~$\frac{1}{2}$. We have omitted this constant in Definition~\ref{defi:0} for convenience, and note that it is enough to add the constant~$2$ in front of the term $\int_0^T\int_{\Omega} D(u) :D(u-\varphi)\;dx\,dt$ in order to find exactly \eqref{eq:0}.\\ 

Finding a solution to the parabolic variational inequality thus amounts to giving meaning to the integral of the nonlinear viscosity coefficient term inherent in the problem, which can be a singular integral in the case of a Bingham fluid.

%%%%%%%%%%%%%%%%%%%%%%%%%%%%%%%%%%%%%%%%%%%%%%%%%%%%%%%%%%%%%%%%%%%%%%%
%%%%%%%%%%%%%%%%%%%%%%%%%%%%%%%%%%%%%%%%%%%%%%%%%%%%%%%%%%%%%%%%%%%%%%%
%%%%%%%%%%%%%%%%%%%%%%%%%%%%%%%%%%%%%%%%%%%%%%%%%%%%%%%%%%%%%%%%%%%%%%%
\section{Main results}
%%%%%%%%%%%%%%%%%%%%%%%%%%%%%%%%%%%%%%%%%%%%%%%%%%%%%%%%%%%%%%%%%%%%%%%
%%%%%%%%%%%%%%%%%%%%%%%%%%%%%%%%%%%%%%%%%%%%%%%%%%%%%%%%%%%%%%%%%%%%%%%
%%%%%%%%%%%%%%%%%%%%%%%%%%%%%%%%%%%%%%%%%%%%%%%%%%%%%%%%%%%%%%%%%%%%%%%

As announced in the introduction, the main result of this paper is the existence of a finite stopping time for the kinetic energy of solutions of the system~\eqref{eq:0}. To this end, we present a proof of the existence of solutions in the form of a variational inequality for this system. The study
of the existence of such solutions has initially been developed in \cite{stampacchia}. Then, this method was
successfully applied for some nonlinear parabolic problems, as the two
dimensional Bingham equations in \cite{duvaut-lions}, or some power
law systems in \cite{jmjnmrmr}. Following a similar approach, we get
the following existence theorem.

\begin{thm}\label{thm:0} Assume that the function~$F$ satisfies the
  hypotheses~(C1)-(C4) and that $\Omega \subset \mathbb{R}^N$,
 ~$N \in \{2,3\}$, is a bounded domain with a Lipschitz boundary, and
  consider an initial datum $u_0 \in H_{0,\sigma}^1(\Omega)$ and a force term
  $f \in L^2((0,T),H^{-1}_\sigma(\Omega))$. Then, there exists a weak
  solution~$u$ of \eqref{eq:0} having the following regularity
  \[
    u \in \mathcal{C}_{w,\mathrm{loc}}\left(\mathbb{R}_+,L^2_{\sigma}(\Omega)\right) \cap
    L^2_{\mathrm{loc}}\left(\mathbb{R}_+, H_{0, \sigma}^1(\Omega)\right) \quad \text{and}
    \quad \partial_tu \in L^{\frac{4}{N}}_{\mathrm{loc}}(\mathbb{R}_+,H^{-1}_\sigma(\Omega)).
  \]
\end{thm}

This result thus ensures the existence of suitable solutions in the two-dimensional and three-dimensional cases. It follows from classical arguments that the solutions are Hölder continuous in time, for a well-chosen Hölder coefficient. 

The nonlinear term in the Bingham equations allows us to obtain the rest of the fluid in finite time in the two-dimensional case. This has been demonstrated in \cite{glowinski}, using the following approach: it is assumed that the force term will compensate the initial kinetic energy of the
 fluid, which amounts to establishing a relation between the norm
~$\lVert u_0 \lVert_{L^2}$ and an integral of
~$\lVert f(t) \lVert_{L^2}$. This argument is based on the use of the
 following two-dimensional Nirenberg-Strauss inequality:

 \[
   \exists \gamma > 0,\; \forall u \in H_0^1(\Omega),\; \lVert u \rVert_{L^2}
   \leq \gamma\int_{\Omega}\lvert D(u) \rvert \; dx.
\]

We note that such an inequality cannot be true in dimension greater than two, because it would contradict the optimality of Sobolev embedding. We therefore propose to slightly adapt this approach to show the existence of a stopping finite time in both the two and the three-dimensional cases. Firstly, let us formalize the definition.

\begin{defi}[Finite stopping time]\label{FST} Let~$u$ be a weak solution in the sense of Definition~\ref{defi:0} of the system~\eqref{eq:0}. We say that~$T_0 \in \mathbb{R}_+$ is a finite stopping time for~$u$ if:
$$\lVert u(T_0) \rVert_{L^2(\Omega)} = 0.$$
\end{defi}

In order to prove the existence of a finite stopping time for the solution~$u$ provided by Theorem~\ref{thm:0}, we do not make any assumption on
the initial velocity field, but we assume that after a certain time
the fluid is no longer subjected to any external force. More exactly
we make some more assumption on~$F$ as stated by the following theorem.

\begin{thm}[Existence of a finite stopping time]
  \label{thm:1}
  Assume that the hypotheses of Theorem~\ref{thm:0} are verified and
  that~$p \in \left[1,2\right)$. Moreover, we assume
  that there exists two positive constants~$\kappa$ and~$T_1$ such that
  \begin{equation}
    \label{eq:F-alpha}
    F(t) \ge \kappa t^{p-2} \text{ for every } t \in (0, +\infty)
    \qquad \text{and} \qquad f = 0 \quad \text{almost everywhere on } (T_1, +\infty).
  \end{equation}
  Then, there exists a finite stopping time 
 ~$T_0 \in \mathbb{R}_+$ for~$u$ in the sense of Definition~\ref{FST}. Moreover, there exists a constant~$C > 0$ such that
  \begin{equation}
      \label{eq:bound-T0}
      T_0 \le T_1 + \frac{C}{1 - s(p)}
      \left(\|u_0\|_{L^2(\Omega)}+ \|f\|_{L^2((0, T_1), H^{-1}_{\sigma}(\Omega))}\right)^{1 - s(p)},
  \end{equation}
  with
\begin{equation}
    \label{eq:s(p)}
    s(p) = \frac{5p - 4}{4 +p}.
\end{equation}
\end{thm}

Thus, this result suggests that if no energy is added to the system, then the kinetic energy associated with an Ostwald-DeWaele or Bingham-type flow should become null in a finite time.
\begin{remark}
Estimate~\eqref{eq:bound-T0} in Theorem~\ref{thm:1} degenerates when~$p \to 2$. More exactly, the upper bound on the stopping time goes to~$+\infty$ when~$p \to 2$. This is related to the loss of the finite stopping time property of the solution in the limit case~$p = 2$.
\end{remark}

%%%%%%%%%%%%%%%%%%%%%%%%%%%%%%%%%%%%%%%%%%%%%%%%%%%%%%%%%%%%%%%%%%%%%%%
%%%%%%%%%%%%%%%%%%%%%%%%%%%%%%%%%%%%%%%%%%%%%%%%%%%%%%%%%%%%%%%%%%%%%%%
%%%%%%%%%%%%%%%%%%%%%%%%%%%%%%%%%%%%%%%%%%%%%%%%%%%%%%%%%%%%%%%%%%%%%%%
\section{Proof of Theorem~\ref{thm:0}}
%%%%%%%%%%%%%%%%%%%%%%%%%%%%%%%%%%%%%%%%%%%%%%%%%%%%%%%%%%%%%%%%%%%%%%%
%%%%%%%%%%%%%%%%%%%%%%%%%%%%%%%%%%%%%%%%%%%%%%%%%%%%%%%%%%%%%%%%%%%%%%%
%%%%%%%%%%%%%%%%%%%%%%%%%%%%%%%%%%%%%%%%%%%%%%%%%%%%%%%%%%%%%%%%%%%%%%%

In this section, we establish the proof of Theorem~\ref{thm:0} in the two-dimensional and three-dimensional settings. In order to prove this result, we begin by establishing an energy estimate for solutions obtained by the Galerkin method in order to obtain uniform bounds with respect to the parameters. We note here that we will have two parameters: a first parameter due to Galerkin's approximation, and a second one due to the regularization proper to the viscosity coefficient~$F$.

First, we briefly establish the Galerkin solutions for the regularized system, with the regularization usually used in numerical methods. Next, we carry out energy estimates to derive, in a third step, weak convergence properties. Finally, we demonstrate the result by making use of properties specific to variational inequalities. 

{\bf First step: Galerkin scheme}

We apply here the usual Galerkin method using the Stokes operator in homogeneous Dirichlet setting, and we use its eigenfunctions $(w_i)_{i \in \mathbb{N}}$ as an orthogonal basis of~$H_{0,\sigma}^1(\Omega)$ and orthonormal basis of~$L^2_{\sigma}(\Omega)$ (see \cite{evans} for details about this property, and \cite[Section 2.3]{rrs} for details concerning the Stokes operator). 

For every positive integer~$m$, we denote by~$P_m $ the projection of~$L^2_{\sigma}(\Omega)$ onto $\text{Span}\left((w_i)_{1 \leq i \leq m}\right)$. We would like to formally define our Galerkin system as follows.
\begin{equation}\label{eq:gal}\left\{
    \begin{array}{ll}
      \partial_tu_{m} + P_m(u_{m}\cdot\nabla u_{m}) +\nabla P_m(\pi) & - \Delta u_{m}
      - P_m\left({\rm{div}}\left(F\left(\lvert D(u_{m}) \rvert\right)
      D(u_{m})\right)\right) = P_mf \\
{\rm{div}}(u_{m}) = 0& \text{ on } \R_+ \times \Omega\\
u_{m} = 0& \text{ on } \R_+ \times \partial\Omega\\
u_{m} = P_m(u_0) & \text{ on } \{0\} \times \Omega.
    \end{array}
  \right.
\end{equation}

In order to avoid the issue posed by the nonlinear term in domains for which the fluid is not deformed we consider the following regularized Galerkin system:

\begin{equation}\label{eq:gal2}
\begin{cases}
\partial_t u_{m,\varepsilon} - P_m\Big({\rm{div}}\Big(F\Big(\sqrt{\varepsilon + \lvert D(u_{m,\varepsilon}) \rvert^2}\Big) D(u_{m,\varepsilon})\Big)\Big) \\
 \phantom{\partial_t u_{m,\varepsilon} - } + P_m(u_{m,\varepsilon} \cdot
  \nabla u_{m,\varepsilon}) + \nabla P_m(\pi) - \Delta u_{m,\varepsilon} = P_mf
& \mathrm{ in }\; \R_+ \times \Omega\\
{\rm{div}}(u_{m,\varepsilon}) = 0 & \mathrm{ in }\; \R_+ \times \Omega\\
u_{m,\varepsilon} = 0 & \mathrm{ on }\; \R_+ \times \partial\Omega\\
{u_{m,\varepsilon}}_{|t=0} = P_m(u_0)& \mathrm{ in }\;\Omega,\\
\end{cases}
\end{equation}

with~$0 < \varepsilon < 1$. Applying a Galerkin method, we can see that, writing $u_{m,\varepsilon}(t) = \sum_{i=1}^md_m^i(t)w_i$, we obtain the ordinary differential system for all~$ 1 \leq i \leq m$:

\begin{multline}\label{eq:galcoeff}
{d^i_m}'(t) =  \langle f,w_i \rangle - \int_{\Omega}\frac{1}{2}\lVert w_i \rVert_{H_0^1}^2d^i_m(t)\; dx - \int_{\Omega}D(u_0):D(w_i)\; dx \\
 - \int_{\Omega}\frac{1}{2}\lVert w_i \rVert^2_{H_0^1}F\left(\sqrt{\varepsilon + \sum_{j=1}^m\frac{1}{2}\lVert w_j \rVert^2_{H_0^1}(d^j_m(t))^2 + 2(D(w_j):D(u_0))d^j_m(t) + \frac{1}{2}\lVert u_0 \rVert^2_{H_0^1}}\right)d^i_m(t)\;dx \\
- \int_{\Omega}F\left(\sqrt{\varepsilon + \sum_{j=1}^m\frac{1}{2}\lVert w_j \rVert^2_{H_0^1}(d^j_m(t))^2 + 2(D(w_j):D(u_0))d^j_m(t) + \frac{1}{2}\lVert u_0 \rVert^2_{H_0^1}}\right)(D(u_0):D(w_i))\;dx \\
- \sum_{j=1}^m\int_{\Omega}w_j \cdot \nabla w_i d^i_m(t)d^j_m(t)\; dx,
\end{multline}
completed with initial condition $d^i_m(0) = (u_0,w_i)_{H_0^1}$. This system is described by a locally Lipschitz continuous function with respect to~$d_m$. Indeed, applying the hypothesis~(C2), the function $\psi : \mathbb{R}^m \to \mathbb{R}$ defined by
\[
  \psi(x) = F\left(\sqrt{\varepsilon^2 + \sum_{j=1}^m\frac{1}{2}\lVert
      w_j \rVert^2_{H_0^1}x_j^2 + 2(D(w_j):D(u_0))x_j +
      \frac{1}{2}\lVert u_0 \rVert^2_{H_0^1}}\right) \qquad \forall x
  \in \mathbb{R}^m
\]
is locally Lipschitz. The Picard-Lindelöf theorem shows the existence of a solution for system~\eqref{eq:gal2} defined on~$\R_+$.

{\bf Second step : Energy estimates and weak convergence}

We recall that the solution~$u_{m,\varepsilon}$ of~\eqref{eq:gal2} belongs to
$\text{Span}\left((w_i)_{1 \leq i \leq m}\right)$, for
$(w_i)_{i \in \mathbb{N}}$ the basis of~$H_{0,\sigma}^1(\Omega)$ which are the
eigenfunctions of the Stokes operator in the homogeneous Dirichlet setting.

In order to clarify our presentation, we specify that we consider the following notion of solution.

\begin{defi}[Solution of \eqref{eq:gal2}]
  \label{defi:1} We say that
  $u_{m,\varepsilon} \in L^2((0, T), H_{0, \sigma}^1(\Omega))$,
  $\partial_tu_{m,\varepsilon} \in L^2((0,T), H^{-1}(\Omega)) $ is a weak
  solution of \eqref{eq:gal2} if for every
  $\varphi \in \mathcal{C}^{\infty}((0,T) \times \Omega)$ and for
  a.e.~$t \in (0,T)$ it satisfies
\begin{equation}\label{eq:wf}
\langle \partial_tu_{m,\varepsilon}, \varphi \rangle + \int_{\Omega} D(u_{m,\varepsilon}):D(\varphi)\; dx + \langle j_{\varepsilon}'(u_{m,\varepsilon}),\varphi \rangle - \int_{\Omega}(u_{m,\varepsilon}\cdot\nabla u_{m,\varepsilon})\cdot\varphi\; dx  = \langle f,\varphi \rangle.
\end{equation}
\end{defi}

We point out that this definition makes sense since we are studying smooth finite dimensional Galerkin solutions. Then, in order to obtain weak limits into the Galerkin formulation, we establish usual energy estimates, proved in Appendix~\ref{app:lemmas}. 

\begin{ppt}\label{ppt:5} Assume that~$u_{m,\varepsilon}$ is a
  solution of \eqref{eq:gal2} in the sense of Definition~\ref{defi:1}.
  Then, there exists a positive constant~$C$ depending on~$p$, $\Omega$, $N$, $\lVert u_0 \rVert_{L^2(\Omega)}$ and~$\lVert f \rVert_{L^2_{\mathrm{loc}}\left(\R_+,H^{-1}(\Omega)\right)}$ such that the
  following estimates hold:
\begin{enumerate}
\item $\lVert u_{m,\varepsilon} \rVert_{L^{\infty}_{\mathrm{loc}}\left(\R_+,L^2_{\sigma}\right)}^2 + \frac{1}{2}\lVert u_{m,\varepsilon} \rVert_{L^2_{\mathrm{loc}}\left(\R_+,H_{0,\sigma}^1\right)}^2 \leq C\left(\lVert f \rVert^2_{L^2_{\mathrm{loc}}\left(\R_+,H^{-1}\right)} + \lVert u_0 \rVert_{L^2}^2\right)$;
\item $\lVert j_{\varepsilon}'(u_{m,\varepsilon})\rVert_{L_{\mathrm{loc}}^{\frac{4}{N}}\left(\R_+,H^{-1}\right)} \leq C\left(1 + \lVert f \rVert_{L^2_{\mathrm{loc}}(\R_+,H^{-1})} + \lVert u_0 \rVert_{L^2}\right)^{p-1}$;

\item $\lVert \partial_tu_{m,\varepsilon}\rVert_{L^{\frac{4}{N}}_{\mathrm{loc}}\left(\R_+,H^{-1}\right)}
  \leq C\left(\lVert f \rVert^2_{L^2_{\mathrm{loc}}\left(\R_+,H^{-1}\right)} + \lVert u_0 \rVert_{L^2}^2\right) + C\left(\lVert f \rVert^2_{L^2_{\mathrm{loc}}\left(\R_+,H^{-1}\right)} + \lVert u_0 \rVert_{L^2}^2\right)^2$
  $$ + C\left(1 + \lVert f \rVert_{L^2_{\mathrm{loc}}(\R_+,H^{-1})} + \lVert u_0 \rVert_{L^2}\right)^{p-1}.$$
\end{enumerate}
\end{ppt}

We then focus in the weak convergence with respect to the above estimates. Here, get suitable convergences by taking the limit with respect to the parameter~$\varepsilon$ in a first time, then by taking the limit with respect to the Galerkin parameter~$m$. Thus, before proving Theorem~\ref{thm:0}, we show suitable weak convergence properties.
 
\begin{lem}
  \label{lem:weak-convergence-eps}
  With the hypotheses of Proposition \ref{ppt:5} there exists
  $v_m \in L^2_{\mathrm{loc}}\left(\R_+,H_{0,\sigma}^1(\Omega)\right) \cap
  L^{\infty}_{\mathrm{loc}}\left(\R_+,L^2_{\sigma}(\Omega)\right)$ with
  $\partial_tv_m \in L^{\frac{4}{N}}_{\mathrm{loc}}\left(\R_+,H^{-1}_{\sigma}(\Omega)\right)$ such that,
  up to subsequences:
\begin{enumerate}
\item $\partial_tu_{m,\varepsilon}\rightharpoonup \partial_tv_{m}$ in $L^{\frac{4}{N}}_{\mathrm{loc}}\left(\R_+,H^{-1}_{\sigma}(\Omega)\right)$;
\item $u_{m,\varepsilon} \rightharpoonup v_{m}$ in $L^2_{\mathrm{loc}}\left(\R_+,H_{0,\sigma}^1(\Omega)\right)$;
\item $u_{m,\varepsilon} \rightarrow v_m$ in $L^2_{\mathrm{loc}}(\R_+,L^2_{\sigma}(\Omega))$;
\item $u_{m,\varepsilon} \overset{*}{\rightharpoonup}v_{m}$ in
  $L^{\infty}_{\mathrm{loc}}\left(\R_+,L^2_{\sigma}(\Omega)\right)$.
\end{enumerate}
Moreover,~$v_m$ satisfies, for every fixed~$T > 0$ and all $\psi \in \mathcal{C}^{\infty}((0,T) \times \Omega)$:
  \begin{align}\label{eq:galm} \frac{1}{2}\left(\lVert v_m(T) \rVert_{L^2}^2 - \frac{1}{2}\lVert u_0 \rVert_{L^2}^2\right) - \int_0^T \langle \partial_tv_m,\psi \rangle\; dt + \int_0^T \int_{\Omega} D(v_m):D(v_m - \psi)\; dx\, dt \nonumber\\
 + \int_0^T j(v_m) - j(\psi)\; dt - \int_0^T \int_{\Omega}(v_m\cdot\nabla v_m)\cdot\psi\; dx\, dt \leq \int_0^T \langle f,v_m - \psi \rangle\; dt.
\end{align}
\end{lem}
\begin{proof}
  
The first and second points follow from the reflexivity of $L^{\frac{4}{N}}_{\mathrm{loc}}\left(\R_+,H^{-1}_{\sigma}(\Omega)\right)$ and $L^2_{\mathrm{loc}}(\R_+,H_{0,\sigma}^1(\Omega))$ respectively, the third one from Aubin-Lions' Lemma, and the last one by Banach-Alaoglu-Bourbaki's theorem.
  
Then, since~$u_{m,\varepsilon}$ is a solution of \eqref{eq:gal2}, it satisfies \eqref{eq:wf}. Testing against $\varphi = u_{m,\varepsilon} - \psi$ in \eqref{eq:wf} for a test function~$\psi$, we have:

\begin{align}\label{eq:9}\langle \partial_tu_{m,\varepsilon}, u_{m,\varepsilon} -\psi \rangle + \int_{\Omega} D(u_{m,\varepsilon}):D(u_{m,\varepsilon} - \psi)\; dx + \langle j_{\varepsilon}'(u_{m,\varepsilon}),u_{m,\varepsilon} - \psi\rangle \nonumber\\
 - \int_{\Omega}(u_{m,\varepsilon}\cdot\nabla u_{m,\varepsilon})\cdot\psi\; dx  = \langle f,u_{m,\varepsilon} - \psi \rangle.
\end{align}

Using \eqref{eq:J-eps-'} leads to the well-known convexity inequality:

\begin{equation}\label{eq:conv}
j_{\varepsilon}(u_{m,\varepsilon}) - j_{\varepsilon}(\psi) \leq \langle j_{\varepsilon}'(u_{m,\varepsilon}),u_{m,\varepsilon} - \psi\rangle.
\end{equation}

Using now Lemma~\ref{lem:liminf-eps} for~$u_{m,\varepsilon}$ in~\eqref{eq:conv}, we get:

$$j(u_{m,\varepsilon}) -C(\varepsilon,u_{m,\varepsilon}) - j_{\varepsilon}(\psi) \leq \langle j_{\varepsilon}'(u_{m,\varepsilon}),u_{m,\varepsilon} - \psi\rangle$$

and then, by (C3) applied to~$u_{m,\varepsilon}$ for the convergence toward~$v_m$, we get:

$$j(u_{m,\varepsilon}) -C(\varepsilon,u_{m,\varepsilon}) - j_{\varepsilon}(\psi) \leq \langle j_{\varepsilon}'(u_{m,\varepsilon}),u_{m,\varepsilon} - \psi\rangle.$$

Then, we can write (see \cite{evans} part 5.9. for details):

\begin{equation}\label{eq:weakL^2}
\forall \varphi \in H_{0,\sigma}^1(\Omega),\; \int_{\Omega}u_{m,\varepsilon}(T)\varphi\; dx = \langle u_{m,\varepsilon}(T),\varphi \rangle = \int_0^T\langle \partial_tu_{m,\varepsilon}(t),\varphi \rangle\; dt + \left\langle u_0, \varphi \right\rangle.
\end{equation}

Now, we also have, using Proposition \ref{ppt:5}:

\begin{align*}
\int_0^T\langle \partial_tu_{m,\varepsilon}(t),\varphi \rangle\; dt + \left\langle u_0, \varphi \right\rangle &\leq \lVert \partial_tu_{m,\varepsilon}\rVert_{L^{\frac{4}{N}}((0,T),H^{-1})}\left(\int_0^T\lVert \varphi \rVert_{H_0^1}^{\frac{4}{4-N}}\; dt\right)^{\frac{4-N}{4}} + C\lVert u_0 \rVert_{L^2}\lVert \varphi \rVert_{H_0^1} \\
&\leq C\left(T^{\frac{4-N}{4}} + \lVert u_0 \rVert_{L^2}\right)\lVert \varphi \rVert_{H_0^1}.
\end{align*}

In the above inequality we considered~$\varphi$ as a function in $L^{\infty}((0,T),H_0^1(\Omega))$, so it belongs to $L^{\frac{4}{4-N}}((0,T),H_0^1(\Omega))$ and its left-hand side defines a linear form over $L^{\frac{4}{N}}((0,T),H^{-1}(\Omega))$. 

Also, the weak convergence leads to:
\begin{equation}\label{eq:weakL2bis}
\int_0^T\langle \partial_tu_{m,\varepsilon}(t),\varphi \rangle\; dt \underset{\varepsilon \rightarrow 0}{\longrightarrow} \int_0^T\langle \partial_tv_m(t),\varphi \rangle\; dt.
\end{equation}

Finally, \eqref{eq:weakL^2} and \eqref{eq:weakL2bis} imply, up to apply a dominated convergence theorem, to:
\begin{equation}\label{eq:weak-epsT}
u_{m,\varepsilon}(T) \underset{\varepsilon \rightarrow 0}{\rightharpoonup} v_m(T)\quad \text{in} \quad L^2(\Omega).
\end{equation}

Then, \eqref{eq:weak-epsT} implies:

\begin{equation}\label{eq:liminfv'v}
\underset{\varepsilon \rightarrow 0}{\rm \underline{lim}}\,\frac{1}{2}\left(\lVert u_{m,\varepsilon}(T) \rVert_{L^2}^2 - \lVert P_m(u_0) \rVert_{L^2}^2\right) \geq \frac{1}{2}\left(\lVert v_m(T) \rVert_{L^2}^2 - \lVert P_m(u_0) \rVert_{L^2}^2\right)
\end{equation}

Also, from usual estimates (see \cite[Chapter 4]{rrs}), since $u_{m,\varepsilon} \underset{\varepsilon \rightarrow 0}{\rightharpoonup} v_m$ in $L^2((0,T),H_{0,\sigma}^1(\Omega))$, we have up to extract:

\begin{equation}\label{eq:10}
\int_0^T \int_{\Omega}\lvert D(u_{m,\varepsilon}) \rvert^2\; dx\,dt \underset{\varepsilon \rightarrow 0}{\longrightarrow} \int_0^T \int_{\Omega}\lvert D(v_m) \rvert^2\; dx\,dt
\end{equation}

and 

\begin{equation}\label{eq:11}
\int_0^T\int_{\Omega}(u_{m,\varepsilon} \cdot \nabla u_{m,\varepsilon})\cdot \psi\;dx\, dt \underset{\varepsilon \rightarrow 0}{\longrightarrow} \int_0^T\int_{\Omega}(v_m \cdot \nabla v_m)\cdot \psi\;dx\, dt.
\end{equation}

Integrating in time \eqref{eq:9}, and passing to the limit over~$\varepsilon$, combining with \eqref{eq:10}, \eqref{eq:11}, \eqref{eq:liminfv'v}  and Lemma~\ref{lem:ae} leads to~\eqref{eq:galm}.
\end{proof}
 
Arguing in the same way, we obtain the following result.

\begin{lem}\label{sec:weak-convergence-m} Under the assumptions of
  Proposition~\ref{ppt:5}, there exists
  $u \in L^2_{\mathrm{loc}}\left(\R_+,H_{0,\sigma}^1(\Omega)\right) \cap
  L^{\infty}_{\mathrm{loc}}\left(\R_+,L^2_{\sigma}(\Omega)\right)$ with $\partial_tu \in L^{\frac{4}{N}}_{\mathrm{loc}}\left(\R_+,H^{-1}_{\sigma}(\Omega)\right)$ such that the function 
 ~$v_m$ given by Lemma~\ref{lem:weak-convergence-eps} verifies.
\begin{enumerate}
\item $\partial_tv_m\rightharpoonup \partial_tu$ in $L^{\frac{4}{N}}_{\mathrm{loc}}\left(\R_+,H^{-1}_{\sigma}(\Omega)\right)$;
\item $v_m \rightarrow u$ in $L^2_{\mathrm{loc}}\left(\R_+,L^2_{\sigma}(\Omega)\right)$;
\item $v_m \rightharpoonup u$ in $L^2_{\mathrm{loc}}(\R_+,H_{0,\sigma}^1(\Omega))$;
\item $v_m \overset{*}{\rightharpoonup}u$ in $L^{\infty}_{\mathrm{loc}}\left(\R_+,L^2_{\sigma}(\Omega)\right)$.
\end{enumerate}
\end{lem}

Moreover, we point out that $u \in C_{w, \mathrm{loc}}(\R_+,L^2_{\sigma}(\Omega))$ from the above estimates (see \cite[Proposition~V.1.7. p.363]{fabrie-boyer} for details). We can now give the proof of Theorem~\ref{thm:0}.

\begin{proof}[Proof of Theorem~\ref{thm:0}]

We take up again the method previously used, that is we write :

\begin{equation}\label{eq:weakL^2'}
\forall \varphi \in H_{0,\sigma}^1(\Omega),\; \int_{\Omega}v_m(T)\varphi\; dx = \langle v_m(T),\varphi \rangle = \int_0^T\langle \partial_tv_m(t),\varphi \rangle\; dt + \left\langle P_m(u_0), \varphi \right\rangle.
\end{equation}

Using Proposition \ref{ppt:5} then leads to:

\begin{align}\label{eq:12}
\int_0^T\langle \partial_tv_m(t),\varphi \rangle\; dt + \left\langle u_0, \varphi \right\rangle &\leq \lVert \partial_tv_m\rVert_{L^{\frac{4}{N}}((0,T),H^{-1})}\left(\int_0^T\lVert \varphi \rVert_{H_0^1}^{\frac{4}{4-N}}\; dt\right)^{\frac{4-N}{4}} + C\lVert u_0 \rVert_{L^2}\lVert \varphi \rVert_{H_0^1} \nonumber \\
&\leq C\left(T^{\frac{4-N}{4}} + \lVert u_0 \rVert_{L^2}\right)\lVert \varphi \rVert_{H_0^1}.
\end{align}

Then, the weak convergence leads to:

\begin{equation}\label{eq:weakL2bis'}
\int_0^T\langle \partial_tv_m(t),\varphi \rangle\; dt \underset{m \rightarrow +\infty}{\longrightarrow} \int_0^T\langle \partial_tu(t),\varphi \rangle\; dt.
\end{equation}

Finally, \eqref{eq:weakL^2'} and \eqref{eq:weakL2bis'} imply:

\begin{equation}\label{eq:weak-epsT'}
v_m(T) \underset{m \rightarrow +\infty}{\rightharpoonup} u(T)\quad \text{in} \quad L^2(\Omega).
\end{equation}
 
Then, \eqref{eq:weak-epsT} implies:

\begin{equation}\label{eq:liminfv'v2}
\underset{m \rightarrow +\infty}{\rm \underline{lim}}\,\frac{1}{2}\left(\lVert v_m(T) \rVert_{L^2}^2 - \lVert P_m(u_0) \rVert_{L^2}^2\right) \geq \frac{1}{2}\left(\lVert u(T) \rVert_{L^2}^2 - \lVert u_0 \rVert_{L^2}^2\right)
\end{equation}

Using once again usual estimates for Navier-Stokes equation, since $v_m \underset{m \rightarrow +\infty}{\longrightarrow} u$ in $L^2((0,T),H_{0,\sigma}^1(\Omega))$, we have:

\begin{equation}\label{eq:10'}
\int_0^T \int_{\Omega}\lvert D(v_m) \rvert^2\; dx\,dt \underset{m \rightarrow +\infty}{\longrightarrow} \int_0^T \int_{\Omega}\lvert D(u) \rvert^2\; dx\,dt
\end{equation}

and 

\begin{equation}\label{eq:11'}
\int_0^T\int_{\Omega}(v_m \cdot \nabla v_m)\cdot \psi\;dx\, dt \underset{m \rightarrow +\infty}{\longrightarrow} \int_0^T\int_{\Omega}(u \cdot \nabla u)\cdot \psi\;dx\, dt.
\end{equation}

Applying lemma \ref{lem:ae} with our assumption (C3) and passing to the limit over~$m$, we get:

\begin{equation}\label{eq:lim-m-j}
\underset{m \rightarrow +\infty}{\rm \underline{lim}}\, \int_0^T j(v_m)\;dt \geq j(u).
\end{equation}

Passing to the limit over~$m$ in \eqref{eq:12}, combining with \eqref{eq:10'}, \eqref{eq:11'}, \eqref{eq:lim-m-j} and \eqref{eq:liminfv'v2} leads to:

\begin{align}\label{eq:sol} \frac{1}{2}\left(\lVert u(T) \rVert_{L^2(\Omega)}^2 - \lVert u_0 \rVert_{L^2(\Omega)}^2\right) -\int_0^T \langle \partial_tu,\psi \rangle\; dt + \int_0^T \int_{\Omega} D(u):D(u - \psi)\; dx\, dt + \int_0^T j(u) - j(\psi)\; dt \nonumber\\
 - \int_0^T \int_{\Omega}(u \cdot\nabla u)\cdot\psi\; dx\, dt \leq \int_0^T \langle f,u - \psi \rangle\; dt
\end{align}
which is the desired result, that is~$u$ is a weak solution of \eqref{eq:0}. 
\end{proof}

%%%%%%%%%%%%%%%%%%%%%%%%%%%%%%%%%%%%%%%%%%%%%%%%%%%%%%%%%%%%%%%%%%%%%%%
%%%%%%%%%%%%%%%%%%%%%%%%%%%%%%%%%%%%%%%%%%%%%%%%%%%%%%%%%%%%%%%%%%%%%%%
%%%%%%%%%%%%%%%%%%%%%%%%%%%%%%%%%%%%%%%%%%%%%%%%%%%%%%%%%%%%%%%%%%%%%%%
\section{Existence of a finite stopping time for shear-thinning flows} \label{sec:finite-stopping-time}
%%%%%%%%%%%%%%%%%%%%%%%%%%%%%%%%%%%%%%%%%%%%%%%%%%%%%%%%%%%%%%%%%%%%%%%
%%%%%%%%%%%%%%%%%%%%%%%%%%%%%%%%%%%%%%%%%%%%%%%%%%%%%%%%%%%%%%%%%%%%%%%
%%%%%%%%%%%%%%%%%%%%%%%%%%%%%%%%%%%%%%%%%%%%%%%%%%%%%%%%%%%%%%%%%%%%%%%

In this part, we assume that hypotheses of the Theorem~\ref{thm:1} are fulfilled. We are interested to show the existence of a finite stopping time of weak solutions of \eqref{eq:0} for a viscosity coefficient~$F$ which behaves at least as a power-law model.%, following classical methods for proving such an extinction profile.

The finite stopping time profile of a flow is specific to the shear-thinning setting, and is a naturally occurring property in many applications. An illustrative example is paints, whose viscosity is expected to decrease as the applied stresses increase, enabling them to spread well, but which are also expected to avoid dripping once the application is complete. To this end, it is expected that the flow will stop rapidly when the fluid is no longer under stress.

From a mathematical point of view, one can observe that nonlinearities proper to Ostwald-DeWaele or Bingham flows in some special cases imply the existence of such a finite stopping time, as it has already been proved for the two-dimensional Bingham equation under some assumptions in~\cite{glowinski}. Moreover, the study of such a profile has been proved in the case of the parabolic~$p$-Laplacian, see \cite[section VII.2]{dibenedetto-degenerate-parabolic} for a bounded initial datum or \cite[Theorem 4.6]{barbu} for the case~$p=1$ and with initial datum belonging to~$L^2(\Omega)$. 

To prove such a result, we proceed to a proof by contradiction. More exactly we show that the regularized Galerkin solutions introduced in the previous section are controlled by a constant~$C(\varepsilon)$ tending to zero as the regularization parameter~$\varepsilon \to 0$. Then, thanks to the energy estimates used previously, we deduce that the solution necessarily stops in finite time, for viscosity coefficients having a form similar to that of~$F(t) = t^{p-2}$, $1 \leq p  < 2$.

In this section, we will moreover assume for convenience that the force term belongs to $L^2_{\mathrm{loc}}(\R_+,L^2(\Omega))$ or, if necessary, we will identify the duality bracket~$\langle\cdot,\cdot\rangle$ with the~$L^2$ inner product. Note that this assumption is not necessary, the results remain valid for $f \in L^2_{\mathrm{loc}}(\R_+,H_{\sigma}^{-1}(\Omega))$. 

Before proving the Theorem~\ref{thm:1}, we need to prove the following useful interpolation lemma, which is a quantified version of J.-L.-Lions' Lemma (see, for instance \cite[Lemme 5.1.]{lions-quelques-resolutions}).

\begin{lem}\label{lem:inequality-extinction} Assume that $u \in L^6(\Omega)$. Then, for all~$r \in (0,3)$, the following inequality holds:

\begin{equation}\label{eq:inequality-extinction}
\lVert u \rVert_{L^2(\Omega)}^{2r} \leq \frac{3-r}{3} \lVert u \rVert_{L^{\frac{3}{2}}(\Omega)}^{\frac{4r}{3-r}} + \frac{r}{3}\lVert u \rVert_{L^6(\Omega)}^2.
\end{equation}

\end{lem}

\begin{proof}[Proof of Lemma~\ref{lem:inequality-extinction}] First, by definition of the~$L^2$-norm, we have, for all~$r>0$:
\begin{equation*}
\lVert u \rVert_{L^2(\Omega)}^{2r} = \Big( \int_{\Omega} \lvert u \rvert^{\frac{4}{3}}  \lvert u \rvert^{\frac{2}{3}}\; dx \Big)^r.
\end{equation*}

The Hölder's inequality in the latter relation leads to

\begin{equation*}
    \| u\|_{L^2(\Omega)}^{2r} \leq \| u\|_{L^{\frac{3}{2}}(\Omega)}^{\frac{4r}{3}}\|u\|_{L^6(\Omega)}^{\frac{2r}{3}},
\end{equation*}

and the usual Young's inequality

\begin{equation*}
    xy \leq \frac{3-r}{3}x^{\frac{3}{3-r}} + \frac{r}{3}y^{\frac{3}{r}},
\end{equation*}

holds for all~$x,y>0$ and~$r\in (0,3)$ directly implies the inequality~\eqref{eq:inequality-extinction}.
\end{proof}

Moreover, we recall the Nirenberg-Strauss inequality:

\begin{lem}[{\cite[Theorem 1]{strauss-71}}]\label{thm:Nirenberg-Strauss} Let~$\Omega$ be an open bounded subset of~$\mathbb{R}^N$ with Lipschitz boundary, then there exists a constant~$C > 0$ which depends of~$N$ and~$\Omega$ such that for all  $u \in W_0^{1,\frac{N}{N-1}}(\Omega)$ the following inequality holds:

\begin{equation}\label{eq:nirenberg-strauss}
\lVert u \rVert_{L^{\frac{N}{N-1}}(\Omega)} \leq C \lVert D(u) \rVert_{L^1(\Omega)}.
\end{equation}
\end{lem}

We are now able to prove Theorem~\ref{thm:1}. We point out that the proof being well-known in the two-dimensional case (see \cite{glowinski})  and can be in that last case a direct application of the Korn's inequality and Sobolev's embedding theorem. For this reason, we only give a proof in the three-dimensional setting.

\begin{proof}[Proof of Theorem~\ref{thm:1}]

Let~$u_{m,\epsilon}$ be the solution of \eqref{eq:gal2}. Choosing~$\varphi = u_{m,\varepsilon}$ in \eqref{eq:wf} we get:

\begin{equation}\label{eq:eq-vm-vm}
  \langle u_{m,\varepsilon}', u_{m,\varepsilon} \rangle 
  +\int_{\Omega} \lvert D(u_{m,\varepsilon}) \rvert^2 \; dx 
  + \langle j_{\varepsilon}'(u_{m,\varepsilon}),u_{m,\varepsilon} \rangle - \underbrace{\int_{\Omega} \left(u_{m,\varepsilon} \nabla u_{m,\varepsilon}\right)u_{m,\varepsilon} \; dx}_{=0} = \langle f, u_{m,\varepsilon} \rangle. 
\end{equation}

Combining~\eqref{eq:J-eps-'} and~\eqref{eq:F-alpha}, we obtain
\[
\langle j_{\varepsilon}'(u_{m,\varepsilon}),u_{m,\varepsilon} \rangle \geq \kappa^{p-2} \int_{\Omega} \lvert D(u_{m,\varepsilon}) \rvert^2\left(\varepsilon + \lvert D(u_{m,\varepsilon})\rvert^2\right)^{\frac{p-2}{2}}\; dx.
\]
Using $\lvert D(u_{m,\varepsilon}) \rvert^2 = (\varepsilon + \lvert D(u_{m,\varepsilon}) \rvert^2) - \varepsilon$ we write
\[
\langle j_{\varepsilon}'(u_{m,\varepsilon}),u_{m,\varepsilon} \rangle
\geq
\kappa^{p-2} \int_{\Omega} \left(\varepsilon + \lvert D(u_{m,\varepsilon})\rvert^2\right)^{\frac{p}{2}}\; dx - 
\kappa^{p - 2} \int_{\Omega} \varepsilon \left(\varepsilon + \lvert D(u_{m,\varepsilon})\rvert^2\right)^{\frac{p-2}{2}} \; dx.
\]

Since $1\leq p<2$, we deduce

\begin{align}
\langle j_{\varepsilon}'(u_{m,\varepsilon}),u_{m,\varepsilon} \rangle
& \geq
\kappa^{p-2} \int_{\Omega} \lvert D(u_{m,\varepsilon})\rvert^p\; dx - 
\kappa^{p - 2} \varepsilon\int_{\Omega}  \varepsilon^{\frac{p-2}{2}} \; dx \nonumber\\
& \geq \kappa^{p-2} \|D(u_{m, \varepsilon})\|_{L^p(\Omega)}^p - 
\kappa^{p-2} \varepsilon^{\frac{p}{2}}|\Omega|. \label{eq:j-estimate}
\end{align}

From \eqref{eq:eq-vm-vm} and~\eqref{eq:j-estimate}, we get:

\begin{equation}\label{eq:est-ext}
\frac{1}{2}\frac{d}{dt}\left(\lVert u_{m,\varepsilon}(t)\rVert_{L^2(\Omega)}^2\right) + \lVert D(u_{m,\varepsilon}) \rVert_{L^{2}(\Omega)}^2 + \kappa^{p -2} \lVert D(u_{m,\varepsilon}) \rVert_{L^{p}(\Omega)}^{p} \leq \langle f,u_{m,\varepsilon} \rangle + \kappa^{p -2} \lvert \Omega \rvert \varepsilon^{\frac{p}{2}}.
\end{equation}

Then, using successively the embedding $L^{p}(\Omega) \hookrightarrow L^1(\Omega)$, assumption~\eqref{eq:F-alpha} and the Lemma~\ref{thm:Nirenberg-Strauss}, we get from 
\eqref{eq:est-ext}, for $t \geq T_1$:
\begin{equation}\label{eq:est-ext-2}
\frac{1}{2}\frac{d}{dt}\left(\lVert u_{m,\varepsilon}(t)\rVert_{L^2(\Omega)}^2\right) + \lVert D(u_{m,\varepsilon}) \rVert_{L^2(\Omega)}^2 + C\lVert u_{m,\varepsilon} \rVert_{L^{\frac{3}{2}}(\Omega)}^{p} \leq \kappa^{p -2} \lvert \Omega \rvert \varepsilon^{\frac{p}{2}}.
\end{equation}

Now, from the embedding $\left\{u \in H_0^1(\Omega) / \lVert D(u) \rVert_{L^2(\Omega)} < +\infty\right\} \hookrightarrow L^6(\Omega)$ which can be obtained using Korn's $L^2$ equality and Sobolev embedding $H_0^1(\Omega) \hookrightarrow L^6(\Omega)$ we get from \eqref{eq:est-ext-2}:

\begin{equation}\label{eq:est-ext-3}
\frac{1}{2}\frac{d}{dt}\left(\lVert u_{m,\varepsilon}(t)\rVert_{L^2(\Omega)}^2\right) + C \lVert u_{m,\varepsilon} \rVert_{L^6(\Omega)}^2 + C \lVert u_{m,\varepsilon} \rVert_{L^{\frac{3}{2}}(\Omega)}^{p} \leq \kappa^{p-2} \lvert \Omega \rvert \varepsilon^{\frac{p}{2}}.
\end{equation}

Applying Lemma~\ref{lem:inequality-extinction} with~$r = \frac{3p}{4+p}$ (which satisfies~$\frac{3}{5}\leq r <1$ since~$1 \leq p < 2$), we have
\begin{align}\label{eq:inequality-extinction-alpha}
    \|u_{m,\varepsilon} \|_{L^2(\Omega)}^{1 + \frac{5p-4}{4+p}} 
    & 
    \leq \frac{4}{4+p}\|u_{m,\varepsilon}\|_{L^{\frac{3}{2}}(\Omega)}^p + \frac{p}{4+p}\| u_{m,\varepsilon} \|_{L^6(\Omega)}^2 \nonumber \\
    & \leq \|u_{m,\varepsilon}\|_{L^{\frac{3}{2}}(\Omega)}^p + \| u_{m,\varepsilon} \|_{L^6(\Omega)}^2.
\end{align}

Note that the exponent~$s(p) := \frac{5p -4}{4+p}$ in the left-hand side is positive since we have chosen~$p \geq 1$. Combining~\eqref{eq:inequality-extinction-alpha} with \eqref{eq:est-ext-3} we deduce that there exists a constant~$C_0>0$ (which does not depend on~$\varepsilon$) such that, for~$t \geq T_1$:

\begin{equation}\label{eq:est-ext-4}
\frac{1}{2}\frac{d}{dt}\left(\lVert u_{m,\varepsilon}(t)\rVert_{L^2(\Omega)}^2\right) + C_0\lVert u_{m,\varepsilon} \rVert_{L^2(\Omega)}^{1 +s(p)} \leq \kappa^{p-2} \lvert \Omega \rvert \varepsilon^{\frac{p}{2}}.
\end{equation}

Assume that for all~$t \geq T_1$ we have $C_0\lVert u_{m,\varepsilon}(t) \rVert_{L^2(\Omega)}^{1 +s(p)} \geq 2\kappa^{p-2} \lvert \Omega \rvert \varepsilon^{\frac{p}{2}}$. Then, we can write from \eqref{eq:est-ext-4}:

\begin{equation}\label{eq:est-ext-5}
\frac{1}{2}\frac{d}{dt}\left(\lVert u_{m,\varepsilon}(t)\rVert_{L^2(\Omega)}^2\right)
\leq
-\frac{C_0}{2}\lVert u_{m,\varepsilon} \rVert_{L^2(\Omega)}^{1 +s(p)}.
\end{equation}

Then dividing by $\lVert u_{m,\varepsilon}(t) \rVert_{L^2(\Omega)}^{1 +s(p)}$ both sides of~\eqref{eq:est-ext-5}, we obtain for all $t \geq T_1$:

\begin{equation}\label{eq:est-ext-6}
\frac{d}{dt}\left(\lVert u_{m,\varepsilon}(t) \rVert_{L^2(\Omega)}^{1-s(p) }\right) \leq -\frac{C_0}{2}(1-s(p)).
\end{equation}

Note that $s(p) < 1$ since $p < 2$. Integrating~\eqref{eq:est-ext-6} with respect to the time leads to $\lVert u_{m,\varepsilon}(t) \rVert_{L^2(\Omega)}^{1-s(p)} < 0$ for~$t$ large enough. This is a contradiction. Consequently, there exists a time $T_{0,\varepsilon} \geq T_1$ such that $C_0\lVert u_{m,\varepsilon}(T_{0,\varepsilon}) \rVert_{L^2(\Omega)}^{1 +s(p)} \leq 2\kappa^{p-2} \lvert \Omega \rvert \varepsilon^{\frac{p}{2}}$, thus the decay of kinetic energy for smooth Galerkin solutions implies that $C_0\lVert u_{m,\varepsilon}(t) \rVert_{L^2(\Omega)}^{1 +s(p)} \leq 2\kappa^{p-2} \lvert \Omega \rvert \varepsilon^{\frac{p}{2}}$ for all $t \geq T_{0,\varepsilon}$. Moreover, considering~$T_{0,\varepsilon}$ as being the smallest time satisfying such an inequality, we get that for~$t \in [T_1,T_{0,\varepsilon}]$, we have that $C_0\lVert u_{m,\varepsilon}(t) \rVert_{L^2(\Omega)}^{1 +s(p)} \geq 2\kappa^{p-2} \lvert \Omega \rvert \varepsilon^{\frac{p}{2}}$, which means that $\lVert u_{m,\varepsilon}(t) \rVert_{L^2(\Omega)}^{-(1 +s(p))} \leq \frac{C_0}{2\kappa^{p-2} \lvert \Omega \rvert \varepsilon^{\frac{p}{2}}}$. Dividing \eqref{eq:est-ext-4} by $\lVert u_{m,\varepsilon}(t) \rVert_{L^2(\Omega)}^{1 +s(p)}$ for $t \in [T_1,T_{0,\varepsilon}]$ then leads once again to

\begin{equation*}
    \frac{d}{dt}\left(\| u_{m,\varepsilon}(t) \|_{L^2}^{1-s(p)}\right) \leq -\frac{C_0}{2}(1-s(p)),
\end{equation*}

which in turn leads, after integrating over $[T_1,T_{0,\varepsilon}]$, to

\begin{equation*}
  \|u_{m,\varepsilon}(T_{0, \varepsilon})\|_{L^2}^{1 - s(p)}
  - \|u_{m,\varepsilon}(T_1)\|_{L^2}^{1 - s(p)} \le -\frac{C_0}{2}(1-s(p))(T_{0, \varepsilon} - T_1)
\end{equation*}
so that it implies
\begin{equation}\label{eq:estimate-T0epsilon}
T_{0, \varepsilon} \le T_1 + \frac{2 \|u_{m,\varepsilon}(T_1)\|_{L^2}^{1 - s(p)}}{C_0(1 - s(p))} \le T_1 +
\frac{2 \left(\|u_0\|_{L^2}+ \|f\|_{L^2((0, T_1), H^{-1}_{\sigma})}\right)^{1 - s(p)}}{C_0(1 - s(p))}.
\end{equation}

Namely, the sequence $(T_{0,\varepsilon})_{\varepsilon > 0}$ is uniformly bounded following the parameter~$\varepsilon > 0$. Thus, letting~$\varepsilon \rightarrow 0$ leads to the existence of~$T_0$ (which may depends of~$m$) such that $\lVert v_m(t) \rVert_{L^2(\Omega)} = 0$ for almost all $t \in [T_0,+\infty)$ and then $\lVert v_m \rVert_{L^2([T_0,+\infty),L^2(\Omega))} = 0$. The same line of arguments shows the existence of a finite stopping time for~$u$ in the sense of Definition~\ref{FST}. This concludes the proof.
\end{proof}

%%%%%%%%%%%%%%%%%%%%%%%%%%%%%%%%%%%%%%%%%%%%%%%%%%%%%%%%%%%%%%%%%%%%%%%
%%%%%%%%%%%%%%%%%%%%%%%%%%%%%%%%%%%%%%%%%%%%%%%%%%%%%%%%%%%%%%%%%%%%%%%
%%%%%%%%%%%%%%%%%%%%%%%%%%%%%%%%%%%%%%%%%%%%%%%%%%%%%%%%%%%%%%%%%%%%%%%
\appendix
%%%%%%%%%%%%%%%%%%%%%%%%%%%%%%%%%%%%%%%%%%%%%%%%%%%%%%%%%%%%%%%%%%%%%%%
%%%%%%%%%%%%%%%%%%%%%%%%%%%%%%%%%%%%%%%%%%%%%%%%%%%%%%%%%%%%%%%%%%%%%%%
%%%%%%%%%%%%%%%%%%%%%%%%%%%%%%%%%%%%%%%%%%%%%%%%%%%%%%%%%%%%%%%%%%%%%%%
\section{Some examples of viscosity coefficients}\label{sec:some-exampl-funct}
%%%%%%%%%%%%%%%%%%%%%%%%%%%%%%%%%%%%%%%%%%%%%%%%%%%%%%%%%%%%%%%%%%%%%%%
%%%%%%%%%%%%%%%%%%%%%%%%%%%%%%%%%%%%%%%%%%%%%%%%%%%%%%%%%%%%%%%%%%%%%%%
%%%%%%%%%%%%%%%%%%%%%%%%%%%%%%%%%%%%%%%%%%%%%%%%%%%%%%%%%%%%%%%%%%%%%%%

 In this section, we give some examples of functions~$F$ satisfying
 the conditions (C1)-(C4), most of which correspond to models of
 non-Newtonian coherent flows in the physical sense. This is the case
 for quasi-Newtonian fluids such as blood, threshold fluids such as
 mayonnaise, or more generally in the case of polymeric liquids.
 
 \begin{enumerate}

\item Firstly, in order to describe power-law fluids (also known as Ostwald-DeWaele flows), we can consider functions $(F_{p})_{1 < p < 2}$ given by:

\begin{center}
\begin{tabular}{l|l}
&$(0,+\infty) \rightarrow (0,+\infty)$\\
$F_{p}$ :&\\
&$t \longmapsto t^{p-2}$.
\end{tabular}
\end{center}

\item Considering functions $(F_{\mu,p})_{\mu > 0, p \in
    [1,2)}$ of the form

\begin{center}
\begin{tabular}{l|l}
&$(0,+\infty) \rightarrow (0,+\infty)$\\
$F_{\mu,p}$ :&\\
&$t \longmapsto (\mu + t^2)^{\frac{p-2}{2}}$
\end{tabular}
\end{center}
leads to Carreau flows.

\item Cross fluids are obtained by choosing function $(F_{\gamma,p})_{\gamma > 0, p \in [1,2)}$ given by:

\begin{center}
\begin{tabular}{l|l}
&$(0,+\infty) \rightarrow (0,+\infty)$\\
$F_{\gamma,p}$ :&\\
&$t \longmapsto (\gamma + t^{2-p})^{-1}$.
\end{tabular}
\end{center}

\item Another possible choice is to take functions $(F_{p,\beta,\gamma})$ given
\begin{center}
\begin{tabular}{l|l}
&$(0,+\infty) \rightarrow (0,+\infty)$\\
$F_{p,\beta,\gamma}$ :&\\
&$t \longmapsto$ $\left\{\begin{tabular}{ll}$t^{p-2}\text{log}(1+t)^{-\beta}$&$\quad \text{if} \; t \in (0,\gamma]$ \\ & \\ $\text{log}(1+\gamma)^{-\beta}t^{p-2}$&$\quad \text{if}\; t \in (\gamma,+\infty)$\end{tabular}\right.$
\end{tabular}
\end{center}
for $1 < p < 2$ and some $\beta,\gamma>0$ with~$\gamma$ small enough.

\end{enumerate}

%%%%%%%%%%%%%%%%%%%%%%%%%%%%%%%%%%%%%%%%%%%%%%%%%%%%%%%%%%%%%%%%%%%%%%%
%%%%%%%%%%%%%%%%%%%%%%%%%%%%%%%%%%%%%%%%%%%%%%%%%%%%%%%%%%%%%%%%%%%%%%%
%%%%%%%%%%%%%%%%%%%%%%%%%%%%%%%%%%%%%%%%%%%%%%%%%%%%%%%%%%%%%%%%%%%%%%%
\section{Useful lemmas and energy estimates}\label{app:lemmas}
%%%%%%%%%%%%%%%%%%%%%%%%%%%%%%%%%%%%%%%%%%%%%%%%%%%%%%%%%%%%%%%%%%%%%%%
%%%%%%%%%%%%%%%%%%%%%%%%%%%%%%%%%%%%%%%%%%%%%%%%%%%%%%%%%%%%%%%%%%%%%%%
%%%%%%%%%%%%%%%%%%%%%%%%%%%%%%%%%%%%%%%%%%%%%%%%%%%%%%%%%%%%%%%%%%%%%%%

For the sake of clarity, in this appendix, we state and prove some useful results employed for the proof of Theorem~\ref{thm:0}. We begin this appendix with some technical lemmas and, in its second part, we give a proof for Proposition~\ref{ppt:5}.

\subsection{Technical lemmas} 
%%%%%%%%%%%%%%%%%%%%%%%%%%%%%%%%%%%%%%%%%%%%%%%%%%%%%%%%%%%%%%%%%%%%%%%

\begin{lem}\label{lemp:6} Let~$X$ be a Banach space, and~$\gamma \geq \frac{1}{2}$. Then, the following inequality holds:
$$\forall (u,v) \in X^2,\; \lVert u + v \rVert^{\gamma}_X \leq 2^{\left(\gamma - \frac{1}{2}\right)}\left(\lVert u \rVert^{\gamma}_X + \lVert v \rVert^{\gamma}_X\right).$$
\end{lem}
\begin{proof}
  Using the convexity of $t \mapsto t^{2(2-p)}$ and triangle's inequality of the norm, we get:
$$\lVert u + v \rVert^{2\gamma}_X = 2^{2\gamma}\left\lVert \frac{u+v}{2} \right\rVert^{2\gamma}_X \leq 2^{2\gamma -1}\left(\lVert u \rVert^{2\gamma}_X + \lVert v \rVert^{2\gamma}_X\right).$$

 Applying now the well-known inequality: $\forall (a,b) \in [0,+\infty)^2,\; \sqrt{a+b} \leq \sqrt{a} + \sqrt{b}$, we get the result.
\end{proof}

 \begin{lem}\label{lem:liminf-eps} Consider that $\varphi \in L^2_{\mathrm{loc}}(\R_+,H_0^1(\Omega))$, then there exists a constant $C(\varepsilon,\varphi) > 0$ which goes to zero as $\varepsilon$ does, such that the following inequality holds:
\begin{equation}\label{eq:liminf-eps}
  j_{\varepsilon}(\varphi) + C(\varepsilon,\varphi) \geq j(\varphi),
\end{equation}
where $j_\varepsilon$ and $j$ are defined by~\eqref{eq:J-eps}.
\end{lem}
 
\begin{proof} Recalling that the assumption (C3) states that $t \mapsto tF(t)$ is increasing,  we get:
\begin{align*}
j(\varphi) &:= \int_{\Omega}\int_0^{\lvert D(\varphi)\rvert} sF(s)\; ds\, dx\\
& \leq \int_{\Omega}\int_0^{\sqrt{\varepsilon}} sF(s)\; ds\, dx + \int_{\Omega}\int_{\sqrt{\varepsilon}}^{\sqrt{\varepsilon} + \lvert D(\varphi)\rvert} sF(s)\; ds\, dx\\
&\leq \varepsilon\sqrt{\varepsilon}F(\varepsilon)\lvert \Omega \rvert + \int_{\Omega}\int_0^{\sqrt{2\lvert D(\varphi) \rvert\sqrt{\varepsilon} + \lvert D(\varphi)\rvert^2}}sF(\sqrt{\varepsilon + s^2})\; ds\, dx\\
&\leq \underbrace{\varepsilon\sqrt{\varepsilon}F(\varepsilon)\lvert \Omega \rvert + \int_{\Omega}\int_{\lvert D(\varphi) \rvert}^{2^{\frac{1}{2}}\varepsilon^{\frac{1}{4}}\lvert D(\varphi) \rvert^{\frac{1}{2}} + \lvert D(\varphi) \rvert}sF(\sqrt{\varepsilon + s^2})\; ds \,dx}_{:= C(\varepsilon,\varphi)} + j_{\varepsilon}(\varphi),
\end{align*}
which is the wished result.
\end{proof}

\begin{lem}\label{lem:ae} Consider $\Omega$ an open bounded subset of $\mathbb{R}^N$ with Lipschitz boundary, and a sequence $(w_n)_{n \in \mathbb{N}}$ such that there exists a positive constant $C > 0$ satisfying $\lVert w_n \rVert_{L^2_{\mathrm{loc}}(\R_+,H_{0,\sigma}^1(\Omega)} \leq C$. Then, for every fixed $T > 0$ and for almost all $(t,x) \in (0,T) \times \Omega$, the following inequality holds:
\begin{equation*}
\underset{n \rightarrow +\infty}{\underline{\mathrm{lim}}}\; \lvert D(w_n)(t,x) \rvert \geq \lvert D(w)(t,x) \rvert.
\end{equation*}
\end{lem}

\begin{proof}
Firstly, let us recall that Eberlein-\u{S}mulyan theorem leads up to an extraction to $w_n \rightharpoonup w$ in $L^2_{\mathrm{loc}}(\R_+,H_0^1(\Omega))$ then, for every fixed $T > 0$ and all Lebesgue points $t_0 \in (0,T)$ and $x_0 \in \Omega$, for all $\delta > $ and $R > 0$ small enough, we have $w_n \rightharpoonup w$ in $L^2((t_0 - \delta, t_0 + \delta),H^1(B(x_0,R))$. Indeed, we have for all test function $\varphi$ :

$$\int_0^T\int_{\Omega}\nabla w_n \cdot \nabla \varphi \;dt\,dx \underset{n \rightarrow +\infty}{\longrightarrow} \int_0^T\int_{\Omega}\nabla w \cdot \nabla \varphi \;dt\,dx.$$

Hence, we can take $\varphi$, which belongs to $\mathcal{C}^{\infty}_0((t_0 - \delta,t_0 + \delta) \times B(x_0,R))$ (up to arguing by density thereafter), satisfying:

$$\nabla \varphi = \left\{\begin{tabular}{ll}
$\nabla\psi$& on $(t_0 - \delta, t_0 + \delta) \times B(x_0,R)$\\
$0$& on $(0,T) \times \Omega \backslash (t_0 - \delta, t_0 + \delta) \times B(x_0,R)$
\end{tabular}\right.$$

and so this leads to:

$$\int_{t_0- \delta}^{t_0 + \delta}\int_{B(x_0,R)}\nabla w_n \cdot \nabla \psi \;dt\,dx \underset{n \rightarrow +\infty}{\longrightarrow} \int_{t_0 - \delta}^{t_0 + \delta}\int_{B(x_0,R)}\nabla w \cdot \nabla \psi \;dt\,dx.$$

That is $w_n\rightharpoonup w$ in $L^2((t_0 - \delta, t_0 + \delta),H^1(B(x_0,R)))$. Also, from Korn's $L^2$ equality and Lebesgue's differentiation theorem over $(\delta,R)$ after dividing by $2\delta\lvert B(x_0,R)\rvert$, one gets that for every Lebesgue point $(t_0,x_0) \in (0,T) \times \Omega$:

$$\lvert D(w_n(t_0,x_0)) \rvert^2 \leq C$$

Following the same line of arguments, we find that:

$$\underset{n \rightarrow +\infty}{\underline{\mathrm{lim}}}\;\int_{t_0 - \delta}^{t_0 + \delta}\int_{B(x_0,R)} \lvert D(w_n) \rvert^2\;dx \,dt \geq \int_{t_0 - \delta}^{t_0 + \delta}\int_{B(x_0,R)} \lvert D(w) \rvert^2\;dx \,dt.$$

Dividing each side by $2\delta \lvert B(x_0,R) \rvert$, we get:

$$\underset{n \rightarrow +\infty}{\underline{\mathrm{lim}}}\;\fint_{t_0 - \delta}^{t_0 + \delta}\fint_{B(x_0,R)} \lvert D(w_n) \rvert^2\;dx \,dt \geq \fint_{t_0 - \delta}^{t_0 + \delta}\fint_{B(x_0,R)} \lvert D(w) \rvert^2\;dx \,dt$$

then letting $(\delta,R) \rightarrow (0,0)$ leads to the result, after applying a dominated convergence theorem.
\end{proof}

\subsection{Proof of Proposition~\ref{ppt:5}}
%%%%%%%%%%%%%%%%%%%%%%%%%%%%%%%%%%%%%%%%%%%%%%%%%%%%%%%%%%%%%%%%%%%%%%%

We now prove the energy estimates used for the convergence of the nonlinear Galerkin method appearing in the proof of Theorem~\ref{thm:0}. 
% \nc{Pas nécéssaire de rappeler cela: First, we recall the statement of these estimates, which are the subject of Proposition~\ref{ppt:5}.

% \begin{ppt} Assume that $u_{m,\varepsilon}$ is a
%   solution of \eqref{eq:gal2} in the sense of Definition~\ref{defi:1}.
%   Then, there exists a positive constant $C$ depending on $p$,
%   $\Omega$, $N$, $\lVert u_0 \rVert_{L^2(\Omega)}$ and
%   $\lVert f \rVert_{L^2_{\mathrm{loc}}\left(\R_+,H^{-1}(\Omega)\right)}$ such that the
%   following estimates hold:
% \begin{enumerate}
% \item $\lVert u_{m,\varepsilon} \rVert_{L^{\infty}_{\mathrm{loc}}\left(\R_+,L^2_{\sigma}\right)}^2 + \frac{1}{2}\lVert u_{m,\varepsilon} \rVert_{L^2_{\mathrm{loc}}\left(\R_+,H_{0,\sigma}^1\right)}^2 \leq C\left(\lVert f \rVert^2_{L^2_{\mathrm{loc}}\left(\R_+,H^{-1}\right)} + \lVert u_0 \rVert_{L^2}^2\right)$;
% \item $\lVert j_{\varepsilon}'(u_{m,\varepsilon})\rVert_{L_{\mathrm{loc}}^{\frac{4}{N}}\left(\R_+,H^{-1}\right)} \leq C\left(1 + \lVert f \rVert_{L^2_{\mathrm{loc}}(\R_+,H^{-1})} + \lVert u_0 \rVert_{L^2}\right)^{p-1}$;

% \item $\lVert \partial_tu_{m,\varepsilon}\rVert_{L^{\frac{4}{N}}_{\mathrm{loc}}\left(\R_+,H^{-1}\right)}
%   \leq C\left(\lVert f \rVert^2_{L^2_{\mathrm{loc}}\left(\R_+,H^{-1}\right)} + \lVert u_0 \rVert_{L^2}^2\right) + C\left(\lVert f \rVert^2_{L^2_{\mathrm{loc}}\left(\R_+,H^{-1}\right)} + \lVert u_0 \rVert_{L^2}^2\right)^2$
%   $$ + C\left(1 + \lVert f \rVert_{L^2_{\mathrm{loc}}(\R_+,H^{-1})} + \lVert u_0 \rVert_{L^2}\right)^{p-1}.$$
% \end{enumerate}
% \end{ppt}
% }

\begin{proof}[Proof of Proposition~\ref{ppt:5}]
  \begin{enumerate}
    \item[]
  \item 
  Setting $\varphi = u_{m,\varepsilon}$ in the weak formulation, we get:
  \[
    \frac{1}{2}
    \frac{d}{dt}\lVert u_{m,\varepsilon} \rVert_{L^2}^2 + \int_{\Omega}
  \lvert D(u_{m,\varepsilon}) \rvert^2\; dx + \underbrace{\langle
    j_{\varepsilon}'(u_{m,\varepsilon}),u_{m,\varepsilon}
    \rangle}_{\geq 0} -
  \underbrace{\int_{\Omega}(u_{m,\varepsilon}\cdot\nabla
    u_{m,\varepsilon})\cdot u_{m,\varepsilon}\; dx}_{=0}  = \langle
  f,u_{m,\varepsilon} \rangle.
  \]

Using the Korn's $L^2$ equality for divergence free vectors fields, we get

$$\frac{d}{dt}\lVert u_{m,\varepsilon}(t) \rVert_{L^2}^2 +  \lVert u_{m,\varepsilon}(t) \rVert_{H_0^1}^2 \leq 2\left\langle f(t),u_{m,\varepsilon}(t) \right\rangle \leq 2\lVert f(t) \rVert_{H^{-1}}^2 + \frac{1}{2}\lVert u_{m,\varepsilon}(t) \rVert_{H_0^1}^2.$$

Then, integrating on $(0,t)$ we get

\begin{equation}\label{eq:4a}
\lVert u_{m,\varepsilon}(t) \rVert^2_{L^2} + \frac{1}{2}\int_0^t\lVert u_{m,\varepsilon} \rVert^2_{H_0^1}\; dt \leq 2\int_0^t \lVert f \rVert^2_{H^{-1}}\; dt + \lVert u_0 \rVert_{L^2}^2.
\end{equation}

Indeed, we recall that $(P_m(u_0),w_i)_{L^2} = (u_0,P_mw_i)_{L^2} = (u_0,w_i)_{L^2}$, and the conclusion follows. From now on, we will omit to detail this last part which is usual.

\item We have, using Cauchy-Schwarz's inequality and Korn's equality in the divergence free $L^2$ setting:
\begin{align}
\left\langle
  j_{\varepsilon}'(u_{m,\varepsilon}),\varphi\right\rangle &=
                                                                     \int_{\Omega}F\left(\sqrt{\varepsilon
                                                                     +
                                                                     \lvert
                                                                     D(u_{m,\varepsilon})
                                                                     \rvert^2}\right)D(u_{m,\varepsilon}):D(\varphi)\;
                                                                     dx
  \nonumber \\
& \leq \frac{1}{\sqrt{2}}\left(\int_{\Omega}F\left(\sqrt{\varepsilon +
                 \lvert D(u_{m,\varepsilon}) \rvert^2}\right)^2\lvert
                 D(u_{m,\varepsilon}) \rvert^2\;
                 dx\right)^{\frac{1}{2}} \lVert \varphi
                 \rVert_{H_0^1}. \label{eq:est-j'eps}
\end{align}

From hypothesis~(C4), setting $A = \Omega \cap \{ \lvert D(u_{m,\varepsilon}) \rvert \leq t_0 \}$ and $B$ its complement in~$\Omega$, we obtain

\begin{align*}
\int_{\Omega}F\left(\sqrt{\varepsilon +  \lvert D(u_{m,\varepsilon}) \rvert^2}\right)^2\lvert D(u_{m,\varepsilon}) \rvert^2\; dx &= \int_{A}F\left(\sqrt{\varepsilon +  \lvert D(u_{m,\varepsilon}) \rvert^2}\right)^2\lvert D(u_{m,\varepsilon}) \rvert^2\; dx\\
& + \int_{B}F\left(\sqrt{\varepsilon +  \lvert D(u_{m,\varepsilon}) \rvert^2}\right)^2\lvert D(u_{m,\varepsilon}) \rvert^2\; dx.
\end{align*}

Let's estimate these two integrals independently. By assumption~(C3), we have that the application $t \mapsto t^2F\left(\sqrt{\varepsilon + t^2}\right)^2$ is non-decreasing, and we obtain directly:

\begin{align*}
\int_{A}F\left(\sqrt{\varepsilon +  \lvert D(u_{m,\varepsilon}) \rvert^2}\right)^2\lvert D(u_{m,\varepsilon}) \rvert^2\; dx &\leq F\left(\sqrt{\varepsilon +  {t_0}^2}\right)^2{t_0}^2\lvert A \rvert\\
& \leq F\left(\sqrt{\varepsilon +  {t_0}^2}\right)^2{t_0}^2\lvert \Omega \rvert\\
& \leq F\left(\sqrt{1+  {t_0}^2}\right)^2\sqrt{1+ {t_0}^2}\lvert \Omega \rvert\\
& \leq C.
\end{align*}

Then we have, using again~(C4):

\begin{align*}
\int_{B}F\left(\sqrt{\varepsilon +  \lvert D(u_{m,\varepsilon}) \rvert^2}\right)^2\lvert D(u_{m,\varepsilon}) \rvert^2\; dx &\leq K\int_{B}\frac{\lvert D(u_{m,\varepsilon}) \rvert^2}{\left(\varepsilon + \lvert D(u_{m,\varepsilon}) \rvert^2\right)^{2-p}}\; dx\\
& \leq K\int_{B}\lvert D(u_{m,\varepsilon}) \rvert^{2(p-1)}\; dx\\
& \leq K\int_B\lvert \nabla u_{m,\varepsilon} \rvert^{2(p-1)}\; dx\\
&\leq C\lVert u_{m,\varepsilon} \rVert^{2(p-1)}_{H_0^1},
\end{align*}

where we used Jensen's inequality in the concave setting with $t \mapsto t^{p-1}$ in the last line. So, we obtain:
 
\begin{equation}\label{eq:5}
\left(\int_{\Omega}F\left(\sqrt{\varepsilon +  \lvert D(u_{m,\varepsilon}) \rvert^2}\right)^2\lvert D(u_{m,\varepsilon}) \rvert^2\; dx\right)^{\frac{1}{2}} \leq \left(C + C\lVert u_{m,\varepsilon} \rVert^{2(p-1)}_{H_0^1}\right)^{\frac{1}{2}}.
\end{equation}

Thus, combining the inequality~\eqref{eq:est-j'eps}-\eqref{eq:5}, using Lemma \ref{lemp:6} with $\gamma = \frac{2}{N}$, and integrating in time leads to:

$$\lVert j_{\varepsilon}'(u_{m,\varepsilon}) \rVert^{\frac{4}{N}}_{L^{\frac{4}{N}}\left((0,T),H^{-1}\right)} \leq C + C\lVert u_{m,\varepsilon} \rVert^{\frac{4(p-1)}{N}}_{L^{\frac{4(p-1)}{N}}((0,T),H_0^1)}.$$

Then, since $0<\frac{4(p-1)}{N} \leq 2$, we get, using the embedding $L^2 \hookrightarrow L^{\frac{4(p-1)}{N}}$ and Lemma~\ref{lemp:6} with $X := H_0^1$, $q = \frac{4(p-1)}{N}$ and $p=2$ on $\lVert u_{m,\varepsilon} \rVert_{L^{\frac{4(p-1)}{N}}((0,T),H_0^1)}$:

$$\lVert j_{\varepsilon}'(u_{m,\varepsilon}) \rVert^{\frac{4}{N}}_{L^{\frac{4}{N}}\left((0,T),H^{-1}\right)} \leq C + C\lVert u_{m,\varepsilon} \rVert^{\frac{4(p-1)}{N}}_{L^{2}((0,T),H_0^1)}.$$

 Using the first point of the proposition for $t=T$, and since $\frac{4(p-1)}{N} \geq 0$, we get:

$$\lVert j_{\varepsilon}'(u_{m,\varepsilon}) \rVert^{\frac{4}{N}}_{L^{\frac{4}{N}}\left((0,T),H^{-1}\right)} \leq C + C(\lVert f \rVert_{L^{2}((0,T),H^{-1})} + \lVert u_0 \rVert_{L^2})^{\frac{4(p-1)}{N}}.$$

Then, using the exponent $\frac{N}{4}$ on both sides and applying once again Lemma~\ref{lemp:6} with $\gamma = \frac{N}{4}$ on the right-hand side in the inequality above leads us to:

$$\lVert j_{\varepsilon}'(u_{m,\varepsilon}) \rVert_{L^{\frac{4}{N}}\left((0,T),H^{-1}\right)} \leq C +C(\lVert f \rVert_{L^{2}((0,T),H^{-1})} + \lVert u_0 \rVert_{L^2})^{p-1}.$$

This is the wished result. 

\item From the weak formulation~\eqref{eq:wf} we get

  \begin{equation}
     \label{eq:wf-direct}
   \langle \partial_tu_{m,\varepsilon},\varphi \rangle = - \int_{\Omega}
D(u_{m,\varepsilon}):D(\varphi)\; dx - \langle
j_{\varepsilon}'(u_{m,\varepsilon}),\varphi \rangle +
\int_{\Omega}(u_{m,\varepsilon} \cdot \nabla u_{m,\varepsilon}) \cdot \varphi\; dx + \langle f,\varphi \rangle.  
   \end{equation}

Let us point out that

   \begin{equation}
     \label{eq:Dphi}
     \int_{\Omega}D(u_{m,\varepsilon}):D(\varphi)\; dx = \frac{1}{2}\int_{\Omega}\nabla u_{m,\varepsilon}\cdot\nabla \varphi\; dx \leq \frac{1}{2}\lVert u_{m,\varepsilon} \rVert_{H_0^1}\rVert \varphi \rVert_{H_0^1}.
   \end{equation}

Also, from Gagliardo-Nirenberg's inequality, we get the existence of a positive constant $C$ which only depends on $N$ and $\Omega$ such that:
 
\begin{equation}\label{eq:est-u}
\lVert u \rVert_{L^4}^2 \leq C\lVert \nabla
u\rVert_{L^2}^{\frac{N}{2}}\lVert u \rVert_{L^2}^{\frac{4-N}{2}}.
\end{equation}

The latter leads, as for the Navier-Stokes equations:

\begin{equation*} \left\lvert \int_{\Omega}(u_{m,\varepsilon}.\nabla
  u_{m,\varepsilon}).\varphi\; dx \right\rvert \leq C \lVert u_{m,\varepsilon} \rVert_{L^2}^{\frac{4-N}{2}}\lVert
                                   u_{m,\varepsilon}
                                   \rVert_{H_0^1}^{\frac{N}{2}}\lVert
                                   \varphi \rVert_{H_0^1}. \label{eq:est-conv}
\end{equation*}

So, putting~\eqref{eq:Dphi}--\eqref{eq:est-conv} and the second estimate of the Proposition~\ref{ppt:5} in~\eqref{eq:wf-direct}, we obtain
\begin{align*}
\langle \partial_tu_{m,\varepsilon},\varphi \rangle \leq &\; \frac{1}{2}\lVert u_{m,\varepsilon} \rVert_{H_0^1}\lVert \varphi \rVert_{H_0^1} + \lVert j_{\varepsilon}'(u_{m,\varepsilon}) \rVert_{H^{-1}}\lVert \varphi \rVert_{H_0^1} +C \lVert u_{m,\varepsilon} \rVert_{L^2}^{\frac{4-N}{2}}\lVert  u_{m,\varepsilon} \rVert_{H_0^1}^{\frac{N}{2}}\lVert \varphi \rVert_{H_0^1}\\
& + \lVert f \rVert_{H^{-1}}\lVert \varphi \rVert_{H_0^1},
\end{align*}

and therefore

$$\lVert \partial_tu_{m,\varepsilon}(t) \rVert_{H^{-1}} \leq  \; \frac{1}{2}\lVert u_{m,\varepsilon} \rVert_{H_0^1} + \lVert j_{\varepsilon}'(u_{m,\varepsilon}) \rVert_{H^{-1}} + C \lVert u_{m,\varepsilon} \rVert_{L^2}^{\frac{4-N}{2}}\lVert  u_{m,\varepsilon} \rVert_{H_0^1}^{\frac{N}{2}}+ \lVert f \rVert_{H^{-1}}.$$

Now, using the following convexity inequality

$$\forall k \in \mathbb{N},\; \forall (x_i)_{1 \leq i \leq k} \in (0,+\infty)^k,\; \exists C > 0,\; \left(\sum_{i=1}^kx_i\right)^{\frac{4}{N}} \leq C\sum_{i=1}^k x_i^{\frac{4}{N}}$$

we get, after integrating in time an using the the embedding $L^2(\Omega) \hookrightarrow L^{\frac{4}{N}}(\Omega)$ (which is valid since $N \in \{2,3\}$, so that we have $\frac{4}{N} \leq 2$):

\begin{align*} \lVert \partial_tu_{m,\varepsilon} \rVert_{L^{\frac{4}{N}}((0,T),H^{-1})}^{\frac{4}{N}}&\; \leq C\left(\lVert u_{m,\varepsilon} \rVert_{L^2((0,T),H_0^1)}^{\frac{4}{N}} + \lVert j_{\varepsilon}'(u_{m,\varepsilon}) \rVert_{L^{\frac{4}{N}}((0,T),H^{-1})}^{\frac{4}{N}}\right)\\
& + C \lVert u_{m,\varepsilon} \rVert_{L^{\infty}((0,T),L^2)}^{\frac{8-2N}{N}}\lVert  u_{m,\varepsilon} \rVert_{L^2((0,T),H_0^1)}^{\frac{4}{N}}+ C\lVert f \rVert_{L^2((0,T),H^{-1})}^{\frac{4}{N}}.
\end{align*}

Using the previously given convexity inequality and the first and second points of the proposition we obtain the desired result.
\end{enumerate}
\end{proof}

{\bf Conflict of interest declaration}

The authors of this article declare that they have no conflicts of interest whatsoever. 

\nocite{duvaut-lions}
\nocite{evans}
\nocite{demengel}
\nocite{glowinski}

\bibliographystyle{plain}
\bibliography{Bibliographie}

\begin{thebibliography}{10}

\bibitem{abbatiello-feireisl-20}
Anna Abbatiello and Eduard Feireisl.
\newblock On a class of generalized solutions to equations describing
  incompressible viscous fluids.
\newblock {\em Ann. Mat. Pura Appl. (4)}, 199(3):1183--1195, 2020.

\bibitem{aafcpm}
Anna Abbatiello and Paolo Maremonti.
\newblock Existence of regular time-periodic solutions to shear-thinning
  fluids.
\newblock {\em J. Math. Fluid Mech.}, 21(2):Paper No. 29, 14, 2019.

\bibitem{ha}
Herbert Amann.
\newblock Stability of the rest state of a viscous incompressible fluid.
\newblock {\em Arch. Rational Mech. Anal.}, 126(3):231--242, 1994.

\bibitem{barbu}
Viorel Barbu.
\newblock {\em Controllability and stabilization of parabolic equations},
  volume~90 of {\em Progress in Nonlinear Differential Equations and their
  Applications}.
\newblock Birkh\"{a}user/Springer, Cham, 2018.
\newblock Subseries in Control.

\bibitem{becker-80}
Ernst Becker.
\newblock Simple non-newtonian fluid flows.
\newblock {\em Advances in applied mechanics}, 20:177--226, 1980.

\bibitem{lbldmr}
Luigi~C. Berselli, Lars Diening, and Michael Ru\v{z}i\v{c}ka.
\newblock Existence of strong solutions for incompressible fluids with shear
  dependent viscosities.
\newblock {\em J. Math. Fluid Mech.}, 12(1):101--132, 2010.

\bibitem{lbmr}
Luigi~C. Berselli and Michael Ru\v{z}i\v{c}ka.
\newblock Global regularity properties of steady shear thinning flows.
\newblock {\em J. Math. Anal. Appl.}, 450(2):839--871, 2017.

\bibitem{bird}
{R. Byron} Bird, {Robert C.} Armstrong, and Ole Hassager.
\newblock {\em Dynamics of polymeric liquids, Volume 1: Fluid mechanics, 2nd
  Edition}.
\newblock Wiley, 1987.

\bibitem{fabrie-boyer}
Franck Boyer and Pierre Fabrie.
\newblock {\em Mathematical tools for the study of the incompressible
  {Navier}-{Stokes} equations and related models}, volume 183 of {\em Appl.
  Math. Sci.}
\newblock New York, NY: Springer, 2013.

\bibitem{BurczakModenaSzekelyhidi-2021}
Jan Burczak, Stefano Modena, and L{\'a}szl{\'o} Sz{\'e}kelyhidi.
\newblock Non uniqueness of power-law flows.
\newblock {\em Communications in Mathematical Physics}, 388:199--243, 2021.

\bibitem{ChlebickaGwiazdaSwierczewska-GwiazdaWroblewska-Kaminska-2021}
Iwona Chlebicka, Piotr Gwiazda, Agnieszka {\'S}wierczewska-Gwiazda, and Aneta
  Wr{\'o}blewska-Kami{\'n}ska.
\newblock {\em Partial differential equations in anisotropic Musielak-Orlicz
  spaces}.
\newblock Springer, 2021.

\bibitem{coussot}
Philippe Coussot.
\newblock {\em Rh{\'e}ophysique: la mati{\`e}re dans tous ses {\'e}tats}.
\newblock EDP sciences Les Ulis, 2012.

\bibitem{fccg}
Francesca Crispo and Carlo~R. Grisanti.
\newblock On the existence, uniqueness and {$C^{1,\gamma}(\overline\Omega)\cap
  W^{2,2}(\Omega)$} regularity for a class of shear-thinning fluids.
\newblock {\em J. Math. Fluid Mech.}, 10(4):455--487, 2008.

\bibitem{demengel}
Fran\c{c}oise Demengel and Gilbert Demengel.
\newblock {\em Espaces fonctionnels}.
\newblock Savoirs Actuels (Les Ulis). [Current Scholarship (Les Ulis)]. EDP
  Sciences, Les Ulis; CNRS \'{E}ditions, Paris, 2007.
\newblock Utilisation dans la r\'{e}solution des \'{e}quations aux
  d\'{e}riv\'{e}es partielles. [Application to the solution of partial
  differential equations].

\bibitem{glowinski}
Jes{\'u}s~I. D{\'i}az, Roland Glowinski, Giovanna Guidoboni, and Taebeom Kim.
\newblock Qualitative properties and approximation of solutions of {B}ingham
  flows: on the stabilization for large time and the geometry of the support.
\newblock {\em Revista de la real academia de Ciencias exactas, Fisicas y
  Naturales}, 104, 2010.

\bibitem{dibenedetto-degenerate-parabolic}
Emmanuele DiBenedetto.
\newblock {\em Degenerate parabolic equations}.
\newblock Universitext. New York, NY: Springer-Verlag, 1993.

\bibitem{ldmr}
Lars Diening and Michael Ru\v{z}i\v{c}ka.
\newblock Strong solutions for generalized {N}ewtonian fluids.
\newblock {\em J. Math. Fluid Mech.}, 7(3):413--450, 2005.

\bibitem{ldmrjw}
Lars Diening, Michael Ru\v{z}i\v{c}ka, and J\"{o}rg Wolf.
\newblock Existence of weak solutions for unsteady motions of generalized
  {N}ewtonian fluids.
\newblock {\em Ann. Sc. Norm. Super. Pisa Cl. Sci. (5)}, 9(1):1--46, 2010.

\bibitem{duvaut-lions}
G.~Duvaut and J.-L. Lions.
\newblock {\em Les in\'{e}quations en m\'{e}canique et en physique}.
\newblock Travaux et Recherches Math\'{e}matiques, No. 21. Dunod, Paris, 1972.

\bibitem{hemr}
Hannes Eberlein and Michael Ru\v{z}i\v{c}ka.
\newblock Existence of weak solutions for unsteady motions of
  {H}erschel-{B}ulkley fluids.
\newblock {\em J. Math. Fluid Mech.}, 14(3):485--500, 2012.

\bibitem{evans}
Lawrence~C. Evans.
\newblock {\em Partial differential equations}, volume~19 of {\em Graduate
  Studies in Mathematics}.
\newblock American Mathematical Society, Providence, RI, second edition, 2010.

\bibitem{jfjmms}
Jens Frehse, Josef M\'{a}lek, and Mark Steinhauer.
\newblock On analysis of steady flows of fluids with shear-dependent viscosity
  based on the {L}ipschitz truncation method.
\newblock {\em SIAM J. Math. Anal.}, 34(5):1064--1083, 2003.

\bibitem{jfmr}
Jens Frehse and Michael Ru\v{z}i\v{c}ka.
\newblock Non-homogeneous generalized {N}ewtonian fluids.
\newblock {\em Math. Z.}, 260(2):355--375, 2008.

\bibitem{friedman-86}
Avner Friedman.
\newblock Optimal control for variational inequalities.
\newblock {\em SIAM J. Control Optim.}, 24(3):439--451, 1986.

\bibitem{galdi}
Giovanni~P. Galdi, Rolf Rannacher, Anne~M. Robertson, and Stefan Turek.
\newblock {\em Hemodynamical flows}, volume~37 of {\em Oberwolfach Seminars}.
\newblock Birkh\"{a}user Verlag, Basel, 2008.
\newblock Modeling, analysis and simulation, Lectures from the seminar held in
  Oberwolfach, November 20--26, 2005.

\bibitem{glowinski-lions-tremolieres}
Roland Glowinski, Jacques-Louis Lions, and Raymond Tr\'{e}moli\`eres.
\newblock {\em Numerical analysis of variational inequalities}, volume~8 of
  {\em Studies in Mathematics and its Applications}.
\newblock North-Holland Publishing Co., Amsterdam-New York, 1981.
\newblock Translated from the French.

\bibitem{ito-kunisch-09}
Kazufumi Ito and Karl Kunisch.
\newblock Optimal control of parabolic variational inequalities.
\newblock {\em J. Math. Pures Appl. (9)}, 93(4):329--360, 2010.

\bibitem{stampacchia}
David Kinderlehrer and Guido Stampacchia.
\newblock {\em An introduction to variational inequalities and their
  applications}, volume~31 of {\em Classics in Applied Mathematics}.
\newblock Society for Industrial and Applied Mathematics (SIAM), Philadelphia,
  PA, 2000.
\newblock Reprint of the 1980 original.

\bibitem{lions-quelques-resolutions}
J.-L. Lions.
\newblock {\em Quelques m\'{e}thodes de r\'{e}solution des probl\`emes aux
  limites non lin\'{e}aires}.
\newblock Dunod, Paris; Gauthier-Villars, Paris, 1969.

\bibitem{jmjnmrmr}
J.~M\'{a}lek, J.~Ne\v{c}as, M.~Rokyta, and M.~Ru\v{z}i\v{c}ka.
\newblock {\em Weak and measure-valued solutions to evolutionary {PDE}s},
  volume~13 of {\em Applied Mathematics and Mathematical Computation}.
\newblock Chapman \& Hall, London, 1996.

\bibitem{jmjnmr}
J.~M\'{a}lek, J.~Ne\v{c}as, and M.~Ru\v{z}i\v{c}ka.
\newblock On weak solutions to a class of non-{N}ewtonian incompressible fluids
  in bounded three-dimensional domains: the case {$p\geq2$}.
\newblock {\em Adv. Differential Equations}, 6(3):257--302, 2001.

\bibitem{jmjnkr}
Josef M\'{a}lek, Jind\v{r}ich Ne\v{c}as, and K.~R. Rajagopal.
\newblock Global analysis of the flows of fluids with pressure-dependent
  viscosities.
\newblock {\em Arch. Ration. Mech. Anal.}, 165(3):243--269, 2002.

\bibitem{RobertsonSequeiraOwens-2009}
Anne~M. Robertson, Ad\'{e}lia Sequeira, and Robert~G. Owens.
\newblock Rheological models for blood.
\newblock In {\em Cardiovascular mathematics}, volume~1 of {\em MS\&A. Model.
  Simul. Appl.}, pages 211--241. Springer Italia, Milan, 2009.

\bibitem{rrs}
James~C. Robinson, Jos\'{e}~L. Rodrigo, and Witold Sadowski.
\newblock {\em The three-dimensional {N}avier-{S}tokes equations}, volume 157
  of {\em Cambridge Studies in Advanced Mathematics}.
\newblock Cambridge University Press, Cambridge, 2016.
\newblock Classical theory.

\bibitem{saramito}
Pierre Saramito.
\newblock {\em Complex fluids}, volume~79 of {\em Math\'{e}matiques \&
  Applications (Berlin) [Mathematics \& Applications]}.
\newblock Springer, Cham, 2016.
\newblock Modeling and algorithms.

\bibitem{strauss-71}
Monty~J. Strauss.
\newblock Variations of {K}orn's and {S}obolev's equalities.
\newblock In {\em Partial differential equations ({P}roc. {S}ympos. {P}ure
  {M}ath., {V}ol. {XXIII}, {U}niv. {C}alifornia, {B}erkeley, {C}alif., 1971)},
  pages 207--214. Amer. Math. Soc., Providence, R.I., 1973.

\bibitem{jw}
J\"{o}rg Wolf.
\newblock Existence of weak solutions to the equations of non-stationary motion
  of non-{N}ewtonian fluids with shear rate dependent viscosity.
\newblock {\em J. Math. Fluid Mech.}, 9(1):104--138, 2007.

\bibitem{jzzt}
Jianfeng Zhou and Zhong Tan.
\newblock Regularity of weak solutions to a class of nonlinear problem with
  non-standard growth conditions.
\newblock {\em J. Math. Phys.}, 61(9):091509, 23, 2020.

\end{thebibliography}

\end{document}